\documentclass[12pt,draft]{article}

\usepackage[english]{babel}

\usepackage{amssymb}
\usepackage{amsmath}
\usepackage{mathrsfs}
\usepackage{theorem}

\newtheorem{theorem}{Theorem}
\newtheorem{lemma}{Lemma}
\newtheorem{corollary}{Corollary}
\newenvironment{proof} {{\bf Proof.}}{\hfill \fbox{}\\ \smallskip}

{\theorembodyfont{\rmfamily} \newtheorem{remark}{Remark}}
{\theorembodyfont{\rmfamily} }
{\theorembodyfont{\rmfamily} \newtheorem{example}{Example}}

\newcommand{\C}{\mathbb{C}}

\newcommand{\n}{\nu}

\newcommand{\de}{\delta}
\newcommand{\ga}{\gamma}
\newcommand{\ep}{\varepsilon}
\newcommand{\Si}{\Sigma}
\newcommand{\vf}{\varphi}

\newcommand{\g}{\gamma}
\newcommand{\si}{\sigma}

\newcommand{\al}{\alpha}
\newcommand{\be}{\beta}

\newcommand{\lam}{\lambda}
\newcommand{\ze}{\zeta}

\newcommand{\Om}{\Omega}
\newcommand{\om}{\omega}
\newcommand{\ro}{\varrho}
\newcommand{\dive}{\mathop{\rm div}\nolimits}

\newcommand{\la}{\langle}
\newcommand{\ra}{\rangle}

\renewcommand\leq{\leqslant}
\renewcommand\geq{\geqslant}
\newcommand\lan{\langle}
\newcommand\ran{\rangle}

\newcommand\Cspt{C_{0}}

\newcommand\spt{\text{spt}\,}

\newcommand{\R}{\mathbb R}
\newcommand{\CC}{\mathbb{C}}

\newcommand{\Hspt}{\mathaccent"017{H}}

\renewcommand\Re{\mathop{\mathbb R \rm{e}}\nolimits}
\renewcommand\Im{\mathop{\mathbb I \rm{m}}\nolimits}
\newcommand\elle{\mathop{\mathscr L}\nolimits}
\newcommand\A{\mathop{\mathscr A}\nolimits}
\newcommand\B{\mathop{\mathscr B}\nolimits}
\newcommand\Sm{\mathop{\mathscr S}\nolimits}

\newcommand\Dm{\mathop{\mathscr D}\nolimits}

\newcommand{\BB}{\mathscr B}

\title{Criterion for the functional dissipativity of second order differential operators with complex
coefficients}

\author{A. Cialdea
\thanks{Department of Mathematics, Computer Sciences and Economics,
University of Ba\-si\-li\-ca\-ta, V.le dell'Ateneo Lucano, 10, 85100 Potenza, Italy.
 \textit{email:}
cialdea@email.it.}\and
V. Maz'ya 
\thanks{Department of Mathematics, Link\"oping University,
SE-581 83, Link\"oping, Sweden.
RUDN University,
6 Miklukho-Maklay St, Moscow, 117198, Russia.
\textit{email}: vladimir.mazya@liu.se.}
}
\date{}    % optional

\begin{document}

\maketitle

{\small {\bf Abstract} In the present paper 
we consider the Dirichlet problem for the second order differential operator $E=\nabla(\A \nabla)$,
where $\A$ is a matrix with complex valued $L^\infty$ entries. 
We  introduce the concept of dissipativity of $E$ with respect to
a given function $\vf:\R^+ \to \R^+$.
Under the assumption that the $\Im \A$ is symmetric, we prove that the condition $|s\, \vf'(s)| \, | \la\Im \A (x)\, \xi,\xi\ra | 
\leq 2\, \sqrt{\vf(s)\, [s\, \vf(s)]'} 
\, \la \Re \A(x) \, \xi,\xi\ra $ (for almost every $x\in\Om\subset \R^N$ and for any $s>0$, $\xi\in\R^N$) is
necessary and sufficient for the functional dissipativity of $E$.


}

 \section{Introduction}

 \subsection{Historical background}
 
 A linear operator $E$ defined on $D(E)\subset L^p(\Om)$ and with range in $L^p(\Om)$ is said to be $L^p$-dissipative if
 $$
 \Re \int_{\Om}\la Eu, u\ra\, |u|^{p-2}dx \leq 0
 $$
for any $u\in D(E)$.  Here $\Om$ is a domain in $\R^N$ and the functions $u$ are complex valued.
 
 Let $E$ be the scalar second order partial differential operator 
 \begin{equation}\label{eq:introE}
Eu= \nabla(\A \nabla u)
\end{equation}
where  $\A$ is a square matrix whose entries are complex valued $L^{\infty}$-functions.

The question of determining necessary and sufficient conditions for the $L^{p}$-dissipativity  ($1<p<\infty$)
of the Dirichlet problem in a domain $\Om\subset \R^{N}$ for the operator \eqref{eq:introE} was considered 
in our paper \cite{cialmaz}.   It is worthwhile to remark that we do not require ellipticity and we may deal with degenerating matrices.

 In particular we have proved that, if $\Im\A$ is symmetric,  the algebraic
condition
\begin{equation}\label{eq:intromain}
	|p-2| \, |\la \Im \A(x) \xi,\xi\ra | \leq 2 \sqrt{p-1} \la \Re\A (x)
\xi,\xi \ra
\end{equation}
for almost any $x\in \Om$ and for any $\xi\in\R^{N}$ is necessary and sufficient for the $L^{p}$-dissipativity of the Dirichlet problem 
for the  operator \eqref{eq:introE}.

We remark that, if $\Im\A$ is symmetric, \eqref{eq:intromain} is equivalent to
the condition
$$
 \frac{4}{p\,p'}\la \Re \A(x)\xi,\xi\ra + \lan \Re \A(x)\eta,\eta\ra
       -2(1-2/p) \la\Im\A(x)  \xi,\eta \ra \geq 0
$$
for almost any $x\in \Om$  and for any $\xi, \eta\in\R^{N}$.
More generally, if the matrix $\Im\A$ is not symmetric,
the  condition 
\begin{equation}\label{intro:form}
 \frac{4}{p\,p'}\la \Re \A(x)\xi,\xi\ra + \lan \Re \A(x)\eta,\eta\ra
       +2 \la (p^{-1}\Im\A(x) + p'^{-1}\Im\A^{*}(x)) \xi,\eta \ra \geq 0
\end{equation}
for almost any $x\in \Om$  and for any $\xi, \eta\in\R^{N}$  ($p'=p/(p-1)$) is only sufficient for the 
$L^p$-dissipativity.

 Condition \eqref{eq:intromain} can be used 
to obtain the sharp angle of dissipativity of  the operator \eqref{eq:introE}. 
To be more precise, we proved  that $zE\, (z \in \C)$ is $L^{p}$-dissipative if and only if $\vartheta_{-}\leq \arg z \leq \vartheta_{+}$, where $\vartheta_{-}$ and $\vartheta_{+}$ are explicitly given (see \cite{cialmaz2}).

If $\Im \A$ is not symmetric or the operator $E$ contains lower order terms 
 \begin{equation}\label{eq:lowerterms}
 Eu= \nabla(\A \nabla u)
+  { b}\nabla u + \nabla(c u)+au.
\end{equation}
condition \eqref{eq:intromain} is only necessary for $E$ to be $L^p$-dissipative.
However  we gave a 
necessary and sufficient condition
for the $L^{p}$-dissipativity of 
operator \eqref{eq:lowerterms} in $\R^{n}$ for
the particular case of constant coefficients (see \cite{cialmaz}).
 
If  operator \eqref{eq:lowerterms} has smooth coefficients and it is strongly elliptic,
then condition  \eqref{eq:intromain} 
is
necessary and sufficient for the $L^{p}$-quasi-dissipativity of 
$E$, i.e. 
for the
$L^{p}$-dissipativity of $E- \om I$, for a suitable $\om>0$.

We extended these results  to the class
of systems of partial differential operators of the form
\begin{equation}\label{eq:systnoelas}
Au=\partial_{h}(\A^{h}(x)\partial_{h}u)
\end{equation}
where $\A^{h}$ are $m\times m$ matrices whose elements
are complex valued $L^{1}_{{\rm loc}}$ functions  (see \cite{cialmaz2}).
We found that the operator $A$ is $L^{p}$-dissipative
if and only if
\begin{equation}\label{eq:condsystnoelas}
\begin{gathered}
 \Re \lan \A^{h}(x) \lam,\lam\ran-(1-2/p)^{2}\Re\lan 
     \A^{h}(x)\om,\om\ran (\Re \lan\lam,\om\ran)^{2}
     \\
     -
     (1-2/p)\Re(\lan \A^{h}(x)\om,\lam\ran -\lan \A^{h}(x)\lam,\om\ran)
     \Re \lan \lam,\om\ran   \geq 0
\end{gathered}
\end{equation}
 for almost every $x\in\Om$ and for any $\lam,\om\in\C^{m}$, $|\om|=1$,
 $h=1,\ldots,N$.
 We have determined also the angle of dissipativity for such operators.

 In the particular case of  positive real symmetric matrices 
 $\A^{h}$, we
 proved that  $A$ is 
    $L^{p}$-dissipative if and only if
    $$
    \left(\frac{1}{2}-\frac{1}{p}\right)^{2} (\mu_{1}^{h}(x)
	+\mu_{m}^{h}(x))^{2} \leq \mu_{1}^{h}(x)\mu_{m}^{h}(x)
    $$
    almost everywhere, $h=1,\ldots,N$,
    where $\mu_{1}^{h}(x)$ and $\mu_{m}^{h}(x)$ are 
    the smallest and the largest 
    eigenvalues of the matrix $\A^{h}(x)$ respectively. These results, obtained in  \cite{cialmaz2},
were new even for systems of ordinary differential
equations.

Peculiar results have been obtained for the system of  linear elasticity
(see \cite{cialmaz2,cialmaz3})
\begin{equation} \label{opelast}
    Eu=\Delta u + (1-2\nu)^{-1}\nabla \dive u
\end{equation}
($\nu$ being the Poisson ratio, $\nu>1$ or $\nu<1/2$),
which is not of the form 
\eqref{eq:systnoelas}.

In particular, for the planar elasticity, we proved 
 (see \cite{cialmaz2})
that operator \eqref{opelast} is $L^{p}$-dissipative if and only if
 \begin{equation}
     \left(\frac{1}{2}-\frac{1}{p}\right)^{2} \leq 
           \frac{2(\nu-1)(2\nu-1)}{(3-4\nu)^{2}}\, . 
     \label{condvecchiaintro}
 \end{equation}

In  \cite{cialmaz3}
 we showed that condition \eqref{condvecchiaintro}
is necessary for the $L^{p}$-dissipativity of operator
\eqref{opelast} in any dimension, even when the Poisson ratio is not constant.
At the present it is not known if condition 
\eqref{condvecchiaintro} is  also sufficient
for the $L^{p}$-dissipativity of elasticity operator
for $N> 2$, in particular for $N=3$. 
Nevertheless, in the same paper, we gave a more strict 
explicit condition which is sufficient 
for the $L^{p}$-dissipativity of \eqref{opelast}.
Indeed we proved that if
 $$
        (1-2/p)^{2}\leq
\begin{cases}
\displaystyle\frac{1-2\n}{2(1-\n)} & \text{if}\ \n<1/2\\
\\
\displaystyle\frac{2(1-\n)}{1-2\n}& \text{if}\ \n > 1,
\end{cases}
$$
 then  the operator \eqref{opelast} is $L^p$-dissipative.

In \cite{cialmaz3} we gave   necessary and
sufficient conditions for a weighted $L^p$-negativity of the
Dirichlet-Lam\'e operator, i.e.
for the validity of the inequality
\begin{equation}\label{eq:introneg}
    \int_{\Om}(\Delta u + (1-2\nu)^{-1}\nabla \dive u)\,
    |u|^{p-2}u\, \frac{dx}{|x|^\alpha} \leq 0
    \end{equation}
under the condition that the vector $u$ is rotationally invariant, i.e. $u$ depends only
on $\ro=|x|$ and $u_\ro$ is the only nonzero spherical
component of $u$. Namely we showed that \eqref{eq:introneg} holds for any such $u$
belonging to $(\Cspt^{\infty}(\R^{N}\setminus\{0\}))^{N}$ if and only if
$$
  -(p-1)(N+p'-2) \leq \alpha \leq N+p-2.
$$

We have considered also the
$L^p$-positivity  of  the fractional powers of the Laplacian $(-\Delta)^\alpha$ ($0<\alpha<1$)
for any $p\in (1,\infty)$ (see \cite[Section 7.6, pp.230--231]{cialmazb}. 
Specifically we have proved that
\begin{equation}\label{eq:fractional}
 \int_{\R^N} \lan (-\Delta)^\alpha u, u\ran |u|^{p-2} dx
   \geq  \frac{2\, c_\alpha}{p\, p'} \, \Vert |u|^{p/2}\Vert^2_{{\mathcal L}^{\alpha,2}(\R^N)}\, ,
\end{equation}
     for any real valued $u\in \Cspt^\infty(\R^N)$,  where
$$
    c_\alpha = -\pi^{-N/2} 4^{\alpha} \Gamma(\alpha+N/2)/ \Gamma(-\alpha) >0.
$$
and $\Vert v\Vert_{{\mathcal L}^{\alpha,2}}$ is the semi-norm
    $$
    \left(\int_{\R^N}\! \int_{\R^N} |v(x+t)-v(x)|^2 \frac{dxdt}{|t|^{N+2\alpha}}\right)^{1/2}.
    $$

All these results are collected in the monograph \cite{cialmazb} where they are considered in the more general frame 
of semi-bounded operators.

The $L^P$-dissipativity of the matrix operator
$$
Eu=\B^{h}(x)\partial_{h}u + \Dm(x) u\, ,
$$
where $\B^h(x)$  and 
$\Dm(x)$ are matrices with complex valued locally
integrable entries defined in the domain  $\Om$
of $\R^N$ and $u=(u_{1},\ldots,u_{m})$ ($1\leq i,j \leq m,\ 1\leq h 
\leq N$), is the subject of paper \cite{cialmaz4}.

We proved that, if $p\neq 2$, $E$ is $L^{p}$-dissipative if, and only if, 
\begin{equation}
    \B^{h}(x)=b_{h}(x) I \  \text{a.e.},
    \label{eq:0}
\end{equation} 
$b_{h}(x)$ being real valued locally integrable 
functions, and
the inequality
$$
\Re \la (p^{-1} \partial_{h}\B^{h}(x) - \Dm(x)) \zeta,\zeta\ra  \geq 0
$$
holds for any $\zeta\in \CC^{m}$, $|\zeta|=1$ and for almost any $x\in\Om$.
If $p=2$ condition \eqref{eq:0} is replaced
by the more general requirement that the matrices $\B^{h}(x)$ are  
self-adjoint a.e..
On combining this with the results we have previously obtained,
we deduced sufficient conditions for the $L^p$-dissipativity
of certain systems of partial differential operators
of the second order.

Paper \cite{cialmaz5} concerns the ``complex oblique
    derivative'' operator, i.e. the boundary
    operator
\begin{equation}\label{eq:odp}
\lambda\cdot\nabla u =
\frac{\partial u}{\partial x_{N}} +
\sum_{j=1}^{N-1}a_{j}\frac{\partial u}{\partial x_{j}}\, ,
\end{equation}
the coefficients $a_{j}$ being   complex valued $L^{\infty}$ functions
defined on $\R^{N-1}$.
We gave new necessary and, separately, sufficient conditions
for the $L^p$-dissipativity of  operator \eqref{eq:odp}.
In the case of real coefficients we provided a necessary and sufficient condition. 
Specifically we proved that, if $a_{j}$ are real valued, the operator $\lambda\cdot\nabla$ is $L^{p}$-dissipative if and only if
there exists a real vector $\Gamma\in L^2_{\text{loc}}(\R^N)$ such that
$$
    -\partial_{j} a_{j} \, \delta(x_n) \leq \frac{2}{p'}( \dive \Gamma - 
    |\Gamma|^{2})
$$
in the sense of distributions.

In the same paper we have considered also a class of integral
operators which can be
written as
\begin{equation}\label{intro:integralpos}
\int_{\R^{N}}^*[u(x)-u(y)]\, K(dx,dy)
\end{equation}
where the integral has to be understood as a principal value in the sense of Cauchy
and the kernel $K(dx,dy)$ is a Borel positive measure defined 
on $\R^N\times\R^N$ satisfying certain conditions. The class
of operators we considered includes the fractional powers of Laplacian
$(-\Delta)^\alpha$, with $0<\alpha<1$. We establish the $L^p$-positivity
of operator \eqref{intro:integralpos}, extending in this way \eqref{eq:fractional}.

We mention that 
{H\"omberg, Krumbiegel and Rehberg \cite{homb} used some of the techniques introduced in \cite{cialmaz} to show 
the $L^{p}$-dissipativity of a certain operator connected to the problem
of the existence of an optimal control for the heat equation with dynamic boundary condition.

%In studying Ornstein–Uhlenbeck operators, Beyn and  Otten  \cite{beyn1,beyn2} 
%assume a condition, which they call \textit{$L^{p}$-dissipativity condition}, which originates from
%our 	\eqref{eq:condsystnoelas}.		

Beyn and  Otten  \cite{beyn1,beyn2} considered the semilinear system
$$
A \Delta v(x)  +\la Sx, \nabla v(x)\ra + f(v(x))=0, \qquad  x\in R^N,
$$
where $A$ is a $m\times m$ matrix, $S$ is a $N\times N$ skew-symmetric matrix
and $f$ is a sufficiently smooth vector function. Among the assumptions they made, they require
the existence of a constant $\ga_A>0$ such that
$$
|z|^2 \Re \la w, Aw\ra + (p-2) \Re \la w,z\ra \Re \la z, Aw\ra \geq \ga_A |z|^{2} |w|^{2}
$$
for any $z,w\in \C^m$. This condition originates from our \eqref{eq:condsystnoelas}.

The results of \cite{cialmaz}
 allowed  Nittka \cite{nittka} to consider 
the case of partial differential operators with complex coefficients. 

Ostermann and  Schratz \cite{oster} obtained the stability  of a numerical procedure for
solving a certain evolution problem.  
The necessary and sufficient condition \eqref{eq:intromain}  show that
their  result does not require  the contractivity of the corresponding semigroup.
  
 Chill,  Meinlschmidt and Rehberg \cite{chill} used some ideas from \cite{cialmaz}
 in the study of the numerical range of second order elliptic operators with mixed boundary conditions in $L^p$.

 Coming back to scalar operators \eqref{eq:introE},
 let us consider the class of operators such that
the form \eqref{intro:form} is not merely
non-negative, but strictly positive, i.e. there exists $\kappa>0$ such that
\begin{equation}\label{intro:formk}
\begin{gathered}
 \frac{4}{p\,p'}\la \Re \A(x)\xi,\xi\ra + \lan \Re \A(x)\eta,\eta\ra
       +2 \la (p^{-1}\Im\A(x) + p'^{-1}\Im\A^{*}(x)) \xi,\eta \ra \\
        \geq \kappa(|\xi|^{2}+|\eta|^{2})
       \end{gathered}
\end{equation}
for almost any $x\in \Om$  and for any $\xi, \eta\in\R^{N}$. 
The class of operators \eqref{eq:lowerterms} whose principal part satisfies \eqref{intro:formk} and which  could be called \textit{$p$-strongly elliptic}, 
was recently considered by several authors.

Carbonaro and Dragi\v{c}evi\'c \cite{carbonaro,carbonaro2} 
 showed the validity of some so called bilinear embeddings related to 
boundary value problems with different boundary conditions for
second order complex coefficient operators satisfying condition \eqref{intro:formk}.
In a series of papers  \cite{dindos1,dindos2,dindos3,dindos4}
 Dindo\v{s} and Pipher
proved several results concerning the $L^p$ solvability of the Dirichlet problem
for the same class of operators.

Finally we mention that recently Maz'ya  and Verbitsky \cite{mazverb}
gave necessary and sufficient conditions for the accretivity of  a second order  partial differential operator $E$ 
containing lower order terms, in the case of Dirichlet data. We observe that the accretivity of $E$ is
equivalent to the $L^2$-dissipativity of $-E$.

\subsection{Functional dissipativity}
 
 A motivation for the study of $L^p$-dissipativity
comes from the decrease of the norm of solutions
of the Cauchy-Dirichlet problem
 \begin{equation}\label{eq:cauchy}
	 \begin{cases}
u'=Eu \\ 
u(0)=u_{0} \, . 
\end{cases}
\end{equation}

Here $u=u(x,t)$, $x\in\Om\subset \R^{N}$, $t>0$ and $u(x,t)=0$ on $\partial\Om$ in some sense
for $t>0$.
 By formal arguments, we have
\begin{equation}\label{eq:dt}
    \frac{d}{dt} \Vert u(\cdot, 
    t)\Vert_{p}^{\ p}=\frac{d}{dt}\int_{\Om}|u(x,t)|^{p}dx =
    p \Re \int_{\Om}\la \partial_{t}u,u\ra |u|^{p-2}dx,
\end{equation}
and then the inequality
$$
\Re \int_{\Om}\la Eu,u\ra |u|^{p-2}dx \leq 0.
$$
implies
the decrease of the $L^{p}$ norm of the solution of the Cauchy-Dirichlet 
problem \eqref{eq:cauchy}.

  More generally, 
let $\Phi$ be a Young function (a convex positive function such that $\Phi(0)=0$ and 
$\Phi(+\infty)=+\infty$) and
consider  the Orlicz  space of functions $u$ for which
there exists $\al>0$ such that
$$
\int_{\Om}\Phi(\al\, |u|)\, dx < +\infty\, .
$$

For the general theory of Orlicz spaces we refer to Krasnosel'ski\u{\i}, Ruticki\u{\i} \cite{krasno}
and Rao, Ren \cite{rao}.
As in \eqref{eq:dt}, if $u(x,y)$ is a solution of
the Cauchy-Dirichlet problem \eqref{eq:cauchy}, we have the decrease of the 
integrals 
$$
\int_{\Om}\Phi( |u(x,t)|)\, dx 
$$
 if
$$
\Re \int_{\Om} \la Eu,u\ra |u|^{-1}\Phi'( |u|)\, dx \leq 0.
$$

This implies  the decrease of the
Luxemburg norm in the related Orlicz space
$$
\Vert u(\cdot, 
    t)\Vert =
\inf \left\{ \lam >0\ |\ \int_{\Om}\Phi(|u(x,t)|/\lam)\, dx \leq 1 
\right\}.
$$

 The aim of the present paper is to
 find conditions for the positive function $\vf$ defined on $(0,+\infty)$
 to satisfy the inequality
\begin{equation}\label{eq:intro1}
\Re \int_{\Om} \la Eu,u\ra\, \vf(|u|)\, dx \leq 0
\end{equation}
for any complex valued $u$ in a certain class, $\Om$ being a domain in $\R^N$. 
Here $E$ is the scalar  operator \eqref{eq:introE}.

In integrals like \eqref{eq:intro1} the combination $\vf(|u|)u$ in 
the integrand is taken 
to be zero where $u$ vanishes, even if the function $\vf(s)$ is not defined at $s=0$.
If $\vf(t)=t^{p-2}$ ($p>1$) we recover the concept of $L^p$-dissipativity.

We remark that the relation between the function $\vf$ in \eqref{eq:intro1}
and $\Phi$ is
\begin{equation}\label{eq:Phiphi}
\vf(t) = \frac{\Phi'(t)}{t}\quad \iff \quad 
\Phi(t)=\int_{0}^{t} s\, \vf(s)\, ds\, 
\end{equation}
and the convexity of $\Phi$ is equivalent to the increase of $s\, \vf(s)$.

If \eqref{eq:intro1} holds for a general $\vf$,
we say that the operator $E$ is \textit{functional dissipative} or $L^{\Phi}$\textit{-dissipative},
 in analogy with the 
terminology used when  $\vf(t)=t^{p-2}$.

More precisely, in the present paper we consider the partial differential operator
\eqref{eq:introE}
and consider the corresponding sesquilinear form
$$
\elle(u,v) = \int_{\Om}\la \A \nabla u, \nabla v \ra\, dx.
$$

We look for the conditions under which the operator $E$
 is $L^{\Phi}$-dissipative, i.e.
$$
\Re \int_\Om \la \A \nabla u, \nabla(\vf(|u|)\, u)\ra\, dx \geq 0
$$
for any $u\in \Hspt^{1}(\Om)$ such that $\vf(|u|)\, u \in \Hspt^{1}(\Om)$.

In the present paper we have considered Dirichlet problem but in principle this notion could be extended to other 
boundary value conditions.

\subsection{The main result}

Let us formulate the main result of the present paper. 
Under the assumption that the matrix $\Im A$ is symmetric,
we prove that  the operator \eqref{eq:introE} is $L^{\Phi}$-dissipative 
 if and only if
\begin{equation}\label{eq:introcnes}
|s\, \vf'(s)| \, | \la\Im \A (x)\, \xi,\xi\ra | 
\leq 2\, \sqrt{\vf(s)\, [s\, \vf(s)]'} 
\, \la \Re \A(x) \, \xi,\xi\ra 
\end{equation}
for almost every $x\in\Om$ and for any $s>0,  \xi\in\R^{N}$.
The function $\vf$ is a positive function defined on $\R^{+}$ such that $s\,\vf(s)$
is strictly increasing. The precise conditions we require on the function $\vf$ are specified
later (see section \ref{subsec:phi}). If $\Im \A$ is not symmetric 
condition \eqref{eq:introcnes} is only necessary for $E$ to be $L^\Phi$-dissipative.

Condition \eqref{eq:introcnes} is equivalent to
$$
{}[1-\Lambda^{2}(t)] \la\Re \A(x)\, \xi,\xi\ra 
+ \la\Re \A(x)\, \eta,\eta\ra 
+2\, \Lambda(t) \la \Im \A(x)\, \xi,\eta\ra 
\geq 0
$$
for almost every $x\in\Om$ and for any $t>0,  \xi, \eta \in \R^{N}$, where $\Lambda$ is the function defined
by the relation
$$\Lambda\left(s\sqrt{\vf(s)}\right)= - \frac{s\, \vf'(s)}{s\,\vf'(s)+2\, 
\vf(s)}\, .
$$

Note that if $\vf(s)=s^{p-2}$, this function is constant and  $\Lambda(t)=-(1-2/p)$,
$1-\Lambda^{2}(t)=4/(p\, p')$.
As for \eqref{intro:form}, if $\Im \A$ is not symmetric, the condition
\begin{equation}\label{eq:Phise0}
\begin{gathered}
{}[1-\Lambda^{2}(t)] \la\Re \A(x)\, \xi,\xi\ra 
+ \la\Re \A(x)\, \eta,\eta\ra +\\
[1+\Lambda(t)] \la \Im \A(x)\, \xi,\eta\ra 
+
[1-\Lambda(t)] \la \Im \A^{*}(x)\, \xi,\eta\ra
\geq 0
\end{gathered}
\end{equation}
for almost every $x\in\Om$ and for any $t>0,  \xi, \eta \in \R^{N}$,
is only sufficient for the $L^{\Phi}$-dissipativity.

If the principal part of operator \eqref{eq:lowerterms} is such that
the left-hand side of \eqref{eq:Phise0} is not merely non negative but strictly positive, i.e. 
\begin{equation}\label{eq:Phise}
\begin{gathered}
{}[1-\Lambda^{2}(t)] \la\Re \A(x)\, \xi,\xi\ra 
+ \la\Re \A(x)\, \eta,\eta\ra +\\
[1+\Lambda(t)] \la \Im \A(x)\, \xi,\eta\ra 
+
[1-\Lambda(t)] \la \Im \A^{*}(x)\, \xi,\eta\ra
\geq \kappa (|\xi|^2 + |\eta|^2)
\end{gathered}
\end{equation}
for a certain $\kappa>0$ and for almost every $x\in\Om$ and for any $t>0,  \xi, \eta \in \R^{N}$,
we say that the operator $E$ is $\Phi$-strongly elliptic.

\subsection{Structure of the paper}

The present paper is organized as follows.

After the short preliminary Section \ref{sec:prelim}, in Section \ref{sec:phi} we specify the class of functions $\vf$ we are going to consider and 
introduce some related functions.

Section \ref{sec:lemma} is devoted to prove a technical lemma concerning real bilinear forms, which will be used
later, in the proof of the main result.

In Section \ref{sec:funcdiss} we give necessary and sufficient conditions for the $L^{\Phi}$-dis\-si\-pativity. Specifically we prove
the equivalence between the $L^{\Phi}$-dissipativity of the operator $E$  and the positiveness of a certain form in $\Hspt^1(\Om)$.
We remark that a similar result holds also for second order differential operators with lower order terms, in analogy with 
\cite[Lemma 1, p.1070]{cialmaz}. 
This can be proved
with the same technique, but for the sake of simplicity here we have preferred to  avoid such a more general formulation.
The section ends with a lemma concerning $\Phi$-strongly elliptic operators.

The main result concerning
condition \eqref{eq:introcnes} is proved in Section \ref{sec:cnes}. We give also some examples showing that in some cases
only real nonnegative operators are $L^{\Phi}$-dissipative, while in other cases the $L^{\Phi}$-dissipativity
is equivalent to the algebraic condition 
$$
\lam_{0}\,  | \la\Im \A (x)\, \xi,\xi\ra | 
\leq 
 \la \Re \A(x) \, \xi,\xi\ra 
$$
for almost any $x\in\Om$ and for any $\xi\in\R^{N}$, where the constant $\lam_0$ is explicitly determined.

\section{Preliminaries and notations}\label{sec:prelim}

Let $\Om$ be an open set in $\R^{N}$.
As usual, by $\Cspt^\infty(\Om)$ we denote the space of  complex valued $C^{\infty}$ functions
having compact support in $\Om$
and by
 $\Hspt^1(\Om)$ the closure of $\Cspt^\infty(\Om)$ in the norm
 $$
 \int_{\Om}(|u|^2  + |\nabla u|^2)dx, 
 $$
$\nabla u$ being the gradient of the function $u$.

The inner product either in
$\C^{N}$ or in $\C$ is denoted by $\la \cdot, \cdot \ra$  and the bar denotes complex conjugation.

In what follows, $\A$ is a $N\times N$ matrix
function with complex valued entries $a_{hk}\in L^{\infty}(\Om)$,
$\A^{t}$ is its transposed matrix and $\A^{*}$ is its adjoint
matrix, i.e. $\A^{*}=\overline{\A}^{t}$.

%In the present paper we deal with second order partial
%differential operators in divergence form
%\begin{equation}\label{eq:E}
%Eu = \nabla(\A \nabla u)\, .
%\end{equation}

Let $\elle$ be the sesquilinear form
$$
\elle(u,v) = \int_{\Om} \la \A \nabla u, \nabla v\ra\, dx \, .
$$

We say that the operator $E$  is $L^{\Phi}$-dissipative if
\begin{equation}\label{eq:defdiss0}
\Re \int_\Om \la \A \nabla u, \nabla(\vf(|u|)\, u)\, dx \geq 0
\end{equation}
for any $u\in \Hspt^{1}(\Om)$ such that $\vf(|u|)\, u \in \Hspt^{1}(\Om)$.

Here $\vf$ is a positive function defined on $\R^{+}=(0,+\infty)$.
In the next section  we specify the conditions we require on $\vf$.

In the sequel we shall sometimes use the following notations.  Given two functions 
$F$ and $G$ defined on a set $Y$,
writing $|F(y)| \lesssim |G(y)|$ we mean that there exists a 
positive constant  $C>0$ such that $|F(y)| \leq C\, |G(y)|$
for any $y\in Y$.  If $|F(y)| \lesssim |G(y)|$ and $|G(y)| \lesssim |F(y)|$ 
we shall write $F(y) \simeq G(y)$.

\section{The function $\vf$ and related functions}\label{sec:phi}

In this Section we introduce the class of function  $\vf$ 
with respect to which we consider the $L^{\Phi}$-dissipativity.  We also introduce other functions
related to $\vf$ and prove some of their properties.

\subsection{The functions $\vf$ and $\psi$}\label{subsec:phi}

The positive function $\vf$ is required to satisfy the following conditions

\begin{enumerate}
	\item\label{item1} $\vf \in C^{1}((0,+\infty))$;
	\item\label{item2} $(s\, \vf(s))'>0$ for any $s>0$;
	\item the range of the strictly increasing function $s\, \vf(s)$ is  $(0,+\infty)$;	
	\item\label{item4} there exist two positive constants $C_{1}, C_{2}$  and a real number $r>-1$ such that
\begin{equation}\label{eq:condsfi}
C_{1} s^{r}\leq (s\vf(s))' \leq C_{2}\, s^{r}, \qquad s\in (0,s_{0})
\end{equation}
for a certain $s_{0}>0$. If $r=0$ we require more restrictive 
conditions: there exists the finite limit $\lim_{s\to 
0^+}\vf(s)=\vf_{+}(0)>0$ 
and  $\lim_{s\to 0^+}s\, \vf'(s)=0$.
\item\label{item5} 
There exists $s_{1}>s_{0}$ such that 
\begin{equation}
   \vf'(s)\geq 0  \text{ or }  \vf'(s)\leq 0 \qquad \forall\ s\geq s_{1}
   . \label{eq:newcond1}
\end{equation}
\end{enumerate}

The condition \ref{item4} prescribes the behaviour of the function $\vf$ in a neighborhood of the origin,
while \ref{item5} concerns the behaviour for large $s$. 

The function $\vf(s)=s^{p-2}$ ($p>1$) provides an example of such a function.
Other examples can be found at the end of the paper. 
%Other examples are given by $\vf(s)=\Phi'(s)/s$ where  $\Phi(s)=s^p\log s$ or  $\Phi(s)=\exp(s^p)-1$
%($p>1$).

From condition  \ref{item4} it follows that, for any $r>-1$, 
	\begin{equation}\label{eq:condfi}
		\vf(s)  \simeq s^{r}, \qquad s\in (0,s_{0}).
\end{equation}

Let us denote by $t\, \psi(t)$ the inverse function of $s\, \vf(s)$. 
The functions
$$
\Phi(s)=\int_{0}^{s} \si\, \vf(\si)\, d\si, \qquad
\Psi(s)= \int_{0}^{s} \si\, \psi(\si)\, d\si
$$
are conjugate Young functions.

\begin{lemma}\label{lemma:new}
The function $\vf$ satisfies conditions \ref{item1}-\ref{item5}
if and only if the function $\psi$ satisfies
the same conditions with $-r/(r+1)$ instead of $r$.
\end{lemma}
\begin{proof}
Inequalities \eqref{eq:condsfi}
and \eqref{eq:condfi} imply
$$
\psi(t) \simeq t^{-r/(1+r)}, \quad
 (t\, \psi(t))' \simeq
t^{-r/(1+r)}\, \qquad t\in(0,t_{0})
$$
for a certain $t_{0}>0$.

Since
$-r/(1+r)>-1$, the function $\psi$ satisfy the conditions 
\ref{item1}-\ref{item4} with $-r/(1+r)$ instead of $r$. In the particular case $r=0$ this follows 
from the equality $t\,\psi(t)\, \vf[t\, \psi(t)]=t$ ($t>0$), which 
implies 
\begin{gather*}
\lim_{t\to 0^{+}}\psi(t)=\frac{1}{\vf_{+}(0)}>0 \, ,\\
\lim_{t\to 0^{+}}t\, \psi'(t) = \lim_{t\to 0^{+}}(\, (t\, \psi(t))' - 
\psi(t)\, ) = \lim_{s\to 0^{+}}\frac{1}{s\vf'(s)+\vf(s)}- 
\frac{1}{\vf_{+}(0)}=0\, .
\end{gather*}

Since $s\, \vf(s)\, \psi[s\, \vf(s)]=s$, we find
$\psi[s\, \vf(s)]=1/\vf(s)$ and then
\begin{equation}\label{eq:psi'}
\psi'[s\, \vf(s)](s\, \vf(s))'= -\frac{\vf'(s)}{\vf^{2}(s)}\,.
\end{equation}

Keeping in mind condition \ref{item2}, we have that
 the function $\psi'$ satisfies condition 
\eqref{eq:newcond1}
 for $t$ greater than  $t_{1}=s_{1}\vf(s_{1})$, but with an opposite sign. 

The viceversa is now obvious, since $- (-r/(1+r)/(1-r/(1+r))=r$.
\end{proof}

\subsection{Some auxiliary functions}\label{subsec:aux}

The function  $s\sqrt{\vf(s)}$ is strictly 
increasing. 
Let $\ze(t)$  be its inverse, i.e. 
$\ze(t) = \left(s\sqrt{\vf(s)}\right)^{-1}$.
The range of $s\sqrt{\vf(s)}$
is $(0,+\infty)$ and $\ze(t)$ belongs to $C^{1}((0,+\infty))$.

Define
\begin{equation}\label{eq:defHGA}
\quad \Theta(t) = \ze(t)/t; \quad \Lambda(t)=t\, \Theta'(t)/\Theta(t)\, .
\end{equation}

From  \eqref{eq:condfi} it follows that there
exists a constant $K>0$ such that
\begin{equation}\label{eq:dispsi}
\ze(t) \leq K \, t^{2/(2+r)}, \quad \Theta(t)\leq K 
t^{-r/(2+r)}, \qquad t\in(0,t_{0})
\end{equation}
for a certain $t_{0}>0$.

We have also
\begin{equation}\label{eq:H=}
\Theta \left(s\sqrt{\vf(s)}\right)=1/\sqrt{\vf(s)}; \quad
	\Theta'\left(s\sqrt{\vf(s)}\right)=- \frac{\vf'(s)}{\vf(s)\,[s\,\vf'(s)+2\, \vf(s)]}\, .
\end{equation}

Note that condition \ref{item2} implies 
$$
s\,\vf'(s)+2\, \vf(s) >0, \qquad s\in(0,+\infty).
$$

	We can write
\begin{gather}
\Lambda\left(s\sqrt{\vf(s)}\right)= - \frac{s\, \vf'(s)}{s\,\vf'(s)+2\, 
\vf(s)} ,
\label{eq:G=}\\
1-\Lambda\left(s\sqrt{\vf(s)}\right)= 2\frac{s\,\vf'(s)+\, \vf(s)}{s\,\vf'(s)+2\, 
\vf(s)}> 0, \nonumber\\
 1+\Lambda\left(s\sqrt{\vf(s)}\right)= \frac{2\, 
\vf(s)}{s\,\vf'(s)+2\, \vf(s)}>0 \, , \nonumber
\end{gather}
from which it follows
\begin{equation}\label{eq:1-G2}
1-\Lambda^{2}\left(s\sqrt{\vf(s)}\right)= \frac{4\, 
\vf(s)\, (s\,\vf'(s)+\, \vf(s))}{(s\,\vf'(s)+2\, \vf(s))^{2}}\, ,
\end{equation}
and
\begin{equation}\label{eq:G<1}
-1< \Lambda(t) < 1
\end{equation}
for any $t>0$. This, together with \eqref{eq:dispsi}, implies 
\begin{equation}\label{eq:disH'}
|\Theta'(t)| \leq \Theta(t)/t \leq K\, t^{-2(1+r)/(2+r)}
\end{equation}
for $t\in(0,t_{0})$.

Finally we give two equalities we shall use later. 
The first equality in \eqref{eq:H=} can be rewritten as
\begin{equation}\label{eq:H2psi}
\Theta^{2}(t)\, \vf[\ze(t)] =1 \, ,
\end{equation}
for any $t>0$, which leads to
$$
2\, \Theta(t)\, \Theta'(t) \vf[\ze(t)] + \Theta^{2}(t)\, \vf'[\ze(t)]\, 
\ze'(t) =0 
$$
and then
$$
\Theta(t)\, \vf'[\ze(t)] \ze'(t) + \Theta'(t)\, \vf[\ze(t)] = - \Theta'(t)\, 
\vf[\ze(t)] = -\Theta'(t)/\Theta^{2}(t).
$$

Since $\ze'(t)=t\, \Theta'(t) + \Theta(t)$ we have also
\begin{equation}\label{eq:tH'}
\Theta(t)\, \vf'[\ze(t)]\, [t\, \Theta'(t) + \Theta(t)] + \Theta'(t)\, \vf[\ze(t)] =
-\Theta'(t)/\Theta^{2}(t)
\end{equation}
for any $t>0$.

\begin{lemma}
Let $\widetilde{\ze}(t)$ the inverse function of $t\, \sqrt{\psi(t)}$ 
and define, as in \eqref{eq:defHGA},
$$
\quad \widetilde{\Theta} (t) = \widetilde{\ze}(t)/t\,; \quad \widetilde{\Lambda}(t)=t\, 
\widetilde{\Theta}'(t)/\widetilde{\Theta}(t)\, .
$$
We have
\begin{equation}\label{eq:GAtilde}
\widetilde{\Theta} (t) = \frac{1}{\Theta(t)}\, , \qquad 
\widetilde{\Lambda}(t)= - \Lambda(t)
\end{equation}
for any $t>0$.
\end{lemma}
\begin{proof}
The function $t\, \psi(t)$ being the inverse of $s\, \vf(s)$, we can 
write
%$$
%s\, \vf(s)\, \psi[s\, \vf(s)] =s\, ,
%$$ 
%i.e.
\begin{equation}\label{eq:psifi=1}
	\vf(s)\, \psi[s\, \vf(s)]=1\, , \forall\ s>0.
\end{equation}

From this and \eqref{eq:H=} we deduce
$$
\Theta(s\, \sqrt{\vf(s)}) = \frac{1}{\sqrt{\vf(s)}}= \sqrt{\psi[s\, 
\vf(s)]}=
\sqrt{\psi(t)}= \frac{1}{\widetilde{\Theta}(t\, \sqrt{\psi(t)})}
$$
where we have set $t= s\, \vf(s)$. On the other hand, keeping in mind 
\eqref{eq:psifi=1}, we have
\begin{equation}\label{eq:sqrt}
t\, \sqrt{\psi(t)} = s\, \vf(s)\, \sqrt{\psi(s\, \vf(s))} = s\, 
\sqrt{\vf(s)}\, .
\end{equation}

The first equality in \eqref{eq:GAtilde} is proved and
the second one follows at once.
\end{proof}

\subsection{A Lemma concerning Sobolev spaces}

We conclude this Section with the next Lemma which guarantees that
the function $\sqrt{\vf(|u|)}\, u$ belongs to the Sobolev space $H^{1}(\Om)$ or $\Hspt^{1}(\Om)$.

\begin{lemma}\label{lemma:fipsi}
If $u\in H^{1}(\Om)$ ($\Hspt^{1}(\Om)$) is such that $\vf(|u|)\, u\in 
H^{1}(\Om)$ ($\Hspt^{1}(\Om)$), then $\sqrt{\vf(|u|)}\, u$ belongs to $H^{1}(\Om)$ 
($\Hspt^{1}(\Om)$).
\end{lemma}
\begin{proof}
Let us suppose 
 $u, \vf(|u|)\, u \in H^{1}(\Om)$.
 The function $\sqrt{\vf(|u|)}\, u$ belongs to 
$L^{2}(\Om)$, because 
we can write  $\vf(|u|)\, |u|^{2}$ as the product of the
two $L^{2}$ functions $\vf(|u|)\, |u|$ and $|u|$. 

Consider now its gradient. 
Suppose  $r\geq 0$  and $\vf'(s)\geq 0$, for $s\geq s_{1}$ (see \eqref{eq:newcond1}). 
We have
	\begin{gather*}
\nabla(\sqrt{\vf(|u|)}\, u)=
 \sqrt{\vf(|u|)}\, \nabla u +
\left(2\sqrt{\vf(|u|)}\right)^{-1} \vf'(|u|)\, \nabla(|u|)\, u 
\end{gather*}
on the set $\Om_{0}=\{x\in\Om\ |\ u(x)\neq 0\}$.

Let us prove that this gradient belongs to $L^2(\Om_{0})$. We can write
	\begin{gather*}
\int_{\Om_{0}}|\nabla(\sqrt{\vf(|u|)}\, u)|^{2}dx =\\
\left(\int_{0<|u|<s_{0}}+\int_{s_{0}\leq|u|\leq s_{1}} +
\int_{|u|>s_{1}}\right)|\nabla(\sqrt{\vf(|u|)}\, u)|^{2}dx\, . 
\end{gather*}

Observing that  $\vf(|u|)\simeq |u|^{r}$ and  $|\vf'(|u|)|\, |u| =| (\vf(|u|)\, |u|)' - \vf(|u|)| \lesssim |u|^{r} \lesssim 1$ for $|u|<s_{0}$
(see \eqref{eq:condsfi} and \eqref{eq:condfi}), we find
\begin{equation}\label{eq:int1}
\begin{gathered}
\left(\int_{0<|u|<s_{0}}|\nabla(\sqrt{\vf(|u|)}\, u)|^{2}dx\right)^{1/2}\leq \\
\left(\int_{0<|u|<s_{0}}\vf(|u|)\, |\nabla u|^{2}\,dx \right)^{1/2}
+
\left(\int_{0<|u|<s_{0}} \frac{\vf'(|u|)^{2}}{\vf(|u|)}\, |u|^{2}|\nabla 
	|u||^{2}dx\right)^{1/2} \lesssim \\
	\left(\int_{0<|u|<s_{0}} |\nabla u|^{2}\,dx \right)^{1/2}
+
\left(\int_{0<|u|<s_{0}} |\nabla 
	|u||^{2}dx\right)^{1/2}
	\lesssim  \left(\int_{\Om} |\nabla u|^{2} dx\right)^{1/2}.
\end{gathered}
\end{equation}

Concerning the set where  $s_{0}\leq|u|\leq s_{1}$ we have
\begin{equation}\label{eq:int2}
\begin{gathered}
\left(\int_{s_{0}\leq|u|\leq s_{1}}|\nabla(\sqrt{\vf(|u|)}\, u)|^{2}dx\right)^{1/2}\leq \\
\left(\int_{s_{0}\leq|u|\leq s_{1}}\vf(|u|)\, |\nabla u|^{2}\,dx \right)^{1/2}
+
\left(\int_{s_{0}\leq|u|\leq s_{1}} \frac{\vf'(|u|)^{2}}{\vf(|u|)}\, |u|^{2}|\nabla 
	|u||^{2}dx\right)^{1/2} \leq \\
	\left(\max_{s\in[s_{0},s_{1}]}\vf(s)
\int_{s_{0}\leq |u| \leq s_{1}} \!\!\!\!\!\!  |\nabla u|^{2}\,dx  \right)^{1/2}
+
\left(\max_{s\in[s_{0},s_{1}]}\frac{(\vf'(s)\, s)^{2}}{\vf(s)}
	\int_{s_{0}\leq |u| \leq s_{1}} \!\!\!\!\!\!   |\nabla |u||^{2}dx\right)^{1/2}\\
	\lesssim  \left(\int_{\Om} |\nabla u|^{2} dx\right)^{1/2}.
\end{gathered}
\end{equation}

Observe now that
\begin{equation}\label{eq:H11}
\begin{gathered}
|\nabla(\vf(|u|)\, u)|^{2} =
|\vf'(|u|)\, \nabla(|u|)\, u + \vf(|u|)\, \nabla u|^{2}=\\
\vf'(|u|)^{2}|u|^{2}|\nabla(|u|)|^{2} + 2 \,\vf'(|u|)\, \vf(|u|)\, 
\la \nabla(|u|), \Re(\overline{u}\, \nabla u)\ra + \vf^{2}(|u|)\, 
|\nabla u|^{2}=\\
[\vf'(|u|)^{2}|u|^{2}+ 2 \,\vf'(|u|)\, \vf(|u|)\,|u|]\,  
|\nabla(|u|)|^{2} + \vf^{2}(|u|)\, 
|\nabla u|^{2}
\end{gathered}
\end{equation}
on $\Om_{0}$.
Since $\vf'(s)\geq 0$ for $s\geq s_{1}$,
 each term in the last line of \eqref{eq:H11} is non 
negative in $\Om_{0}$. This implies that  
each of these terms is integrable on the set $|u|\geq s_{1}$,  the gradient of 
$\vf(|u|)\, u$ belonging to $L^{2}(\Om)$.

By Cauchy inequality we get
\begin{gather*}
	\int_{|u|> s_{1}}\vf(|u|)\, |\nabla u|^{2}\,dx
	\leq \\
\left(\int_{|u| > s_{1}}\vf^{2}(|u|)\, |\nabla 
u|^{2}\,dx\right)^{1/2}
\left(\int_{|u| > s_{1}} |\nabla 
u|^{2}\,dx\right)^{1/2}
 < +\infty\, ,
\end{gather*}
and 
\begin{gather*}
	\int_{|u|> s_{1}}\frac{\vf'(|u|)^{2}}{\vf(|u|)}\, |u|^{2}|\nabla 
	|u||^{2}dx\leq 
	\frac{1}{M_{0}}\int_{|u|> s_{1}}
	\vf'(|u|)^{2}\,  |u|^{2}|\nabla 
	|u||^{2}dx <+\infty\, ,
\end{gather*}
where  $M_{0}>0$ is  
 such that
 $ \vf(s)\geq M_{0}$ for any  $s\geq s_{1}$.

We have then shown that
\begin{equation}\label{eq:int3}
\int_{|u|> s_{1}}|\nabla(\sqrt{\vf(|u|)}\, u)|^{2}dx < +\infty
\end{equation}

Collecting \eqref{eq:int1}, \eqref{eq:int2} and \eqref{eq:int3} we get
$$
\int_{\Om_{0}}|\nabla(\sqrt{\vf(|u|)}\, u)|^{2}dx < +\infty\, .
$$

Suppose now that  $r\geq 0$  and $\vf'(s)\leq 0$, for $s\geq s_{1}$.
Inequalities \eqref{eq:int1} and \eqref{eq:int2} are still valid. 
In order to estimate $\nabla(\sqrt{\vf(|u|)}\, u)$ on the set where $|u|>s_{1}$
we proceed as follows.

Let us define $w=\vf(|u|)\, u$. We have $u=\psi(|w|)\, w$ and then $w$ and $\psi(|w|)\, w$ belong to $H^{1}(\Om)$. 
On the other hand, since $|w|=\vf(|u|)\, |u|$ and in view of equality \eqref{eq:sqrt}, we can write
$\sqrt{\psi(|w|)}\, |w| = \sqrt{\vf(|u|)}\,|u|$.  Recalling the definition of $w$, this implies
\begin{equation}\label{eq:phipsi}
\sqrt{\psi(|w|)}\, w = \sqrt{\vf(|u|)}\, u\, .
\end{equation}

Since $\psi'(t)\geq 0$ for $t \geq s_{1}\vf(s_1)$ (see \eqref{eq:psi'}), we can rewrite formula
\eqref{eq:H11} replacing $\vf(|u|)\, u$ by $\psi(|w|)\, w$ and deduce - as for \eqref{eq:int3} -
that
$$
\int_{|u|> s_{1}}|\nabla(\sqrt{\vf(|u|)}\, u)|^{2}dx =
\int_{|w|> s_{1}\vf(s_1)}|\nabla(\sqrt{\psi(|w|)}\, w)|^{2}dx < +\infty
$$

We have then proved that, if $r\geq 0$,  the vector
\begin{equation}\label{eq:weakgrad}
\left[ \left(2\sqrt{\vf(|u|)}\right)^{-1} \vf'(|u|)\, \nabla(|u|)\, u 
+ \sqrt{\vf(|u|)}\, \nabla u\right]\, \chi_{\Om_{0}}
\end{equation}
($\chi_{\Om_{0}}$ is the characteristic function of $\Om_{0}$) belongs to $L^{2}(\Om)$. 
Let us show that \eqref{eq:weakgrad}
is the weak gradient of $\sqrt{\vf(|u|)}\,  u$. Let 
$\ep>0$ and define
$$
\vf_{\ep}(t)=
\begin{cases}
  \vf(\ep) & \text{if }   |t|\leq \ep \\ 
 \vf(|t|)  & \text{if }  |t|>\ep. 
\end{cases}\, \qquad
h_{\ep}=\sqrt{\vf_{\ep}(|u|)}\, u\, .
$$

The function $h_{\ep}$ belongs to $H^{1}(\Om)$ and
$$
\nabla h_{\ep}=
\begin{cases}
  \sqrt{\vf(\ep)} \, \nabla u& \text{if }  |u|\leq \ep \\ 
 \nabla(\sqrt{\vf(|u|)} \, u) & \text{if }  |u|>\ep 
\end{cases}
$$
almost everywhere in $\Om$.

For any
$f\in 
\Cspt^{\infty}(\Om)$ and for any $j=1,\ldots, N$ we have
\begin{gather*}
\int_{\Om}h_{\ep}\, \partial_{j}f\, dx = 
- \int_{\Om}f\, \partial_{j}h_{\ep}dx =\\
-\sqrt{\vf(\ep)}\int_{|u|\leq \ep}f\, \partial_{j} u \, dx - 
\int_{|u|>\ep}f\, \partial_{j}[\sqrt{\vf(|u|)}\, u]\, dx\, . 
\end{gather*}

Observing that
\begin{equation}\label{eq:maxvf}
|h_{\ep}| =\sqrt{\vf_{\ep}(|u|)}\, |u| \leq \max \left\{\sqrt{\vf(\ep)}, 
\sqrt{\vf(|u|)}\right\}\, |u| \in 
L^{2}(\Om),
\end{equation}
it follows
\begin{gather*}
\int_{\Om}\sqrt{\vf(|u|)}\, u\, \partial_{j}f\, dx =
\lim_{\ep\to 0^{+}}
\int_{\Om}h_{\ep}\, \partial_{j}f\, dx = 
-\int_{\Om_{0}}f\, \partial_{j}[\sqrt{\vf(|u|)}\, u]\, dx\, .
\end{gather*}
This means that \eqref{eq:weakgrad} is the weak gradient of 
$\sqrt{\vf(|u|)}\, u$ and  then $\sqrt{\vf(|u|)}\, u$ belongs to $H^{1}(\Om)$. 

If $u$ and $\vf(|u|)\, u$ are in $\Hspt^{1}(\Om)$ we shall first prove that
\begin{equation}\label{eq:intxparti}
\int_{\Om}\sqrt{\vf(|u|)}\, u\, \partial_{i}f\, dx = 
-\int_{\Om}f\, \partial_{i}(\sqrt{\vf(|u|)}\, u)\, \chi_{\Om_{0}} dx\, .
\end{equation}
for any $f\in\Cspt^{\infty}(\R^{N})$. 

We can write
\begin{gather*}
\int_{\Om}\sqrt{\vf_{\ep}(|u|)}\, u\, \partial_{i}f\, dx =
\int_{\Om}\vf_{\ep}(|u|)\, u\, 
[\vf_{\ep}(|u|)]^{-\frac{1}{2}}\partial_{i}f\, dx=\\
\int_{\Om}\vf_{\ep}(|u|)\, u\, \left[
\partial_{i}([\vf_{\ep}(|u|)]^{-\frac{1}{2}} f)- f\, 
\partial_{i}([\vf_{\ep}(|u|)]^{-\frac{1}{2}})\right]\, dx\, .
\end{gather*}

Moreover, since $\vf(|u|)\, u$ is in $\Hspt^{1}(\Om)$,
\begin{gather*}
\int_{\Om}\vf_{\ep}(|u|)\, u\, 
\partial_{i}([\vf_{\ep}(|u|)]^{-\frac{1}{2}}f)\, dx =\\
\int_{\Om}[\vf_{\ep}(|u|)- \vf(|u|)]\, u\, 
\partial_{i}([\vf_{\ep}(|u|)]^{-\frac{1}{2}} f)\, dx  +
\int_{\Om}\vf(|u|)\, u\, 
\partial_{i}([\vf_{\ep}(|u|)]^{-\frac{1}{2}}f)\, dx = \\
\frac{1}{\sqrt{\vf(\ep)}}
%\left[\sqrt{\vf(\ep)}\right]^{-1}
\int_{|u|<\ep}[\vf(\ep)- \vf(|u|)]\, u\, 
\partial_{i}f\, dx  -
\int_{\Om}[\vf_{\ep}(|u|)]^{-\frac{1}{2}}f \, \partial_{i}(\vf(|u|)\, u)\, dx =\\
 -
\int_{\Om}[\vf_{\ep}(|u|)]^{-\frac{1}{2}}f \, \partial_{i}(\vf(|u|)\, u)\, dx
+ o(1)=\\
- \int_{|u|>\ep} [\vf(|u|)]^{-\frac{1}{2}}\, f \, \partial_{i}(\vf(|u|)\, u)\, dx
+ o(1).
\end{gather*}

This leads to
\begin{gather*}
\int_{\Om}\sqrt{\vf_{\ep}(|u|)}\, u\, \partial_{i}f\, dx =\\
-\int_{|u|>\ep} f\, \left[[\vf(|u|)]^{-\frac{1}{2}}\, \partial_{i}(\vf(|u|)\, u)
+ \vf(|u|)\, u\, 
\partial_{i} ([\vf(|u|)]^{-\frac{1}{2}}) \right] dx + o(1)=\\
-\int_{|u|>\ep} f\,  \partial_{i}(\sqrt{\vf(|u|)}\, u)\, dx + o(1)\, .
\end{gather*}

In view of \eqref{eq:maxvf},
letting $\ep \to 0^{+}$, we get \eqref{eq:intxparti}.
This means that the function $\sqrt{\vf(|u|)}\, u$ extended by zero outside $\Om$ belongs to $H^{1}(\R^{N})$. 
Now we may appeal to a result proved by Deny \cite{deny} (see also Hedberg \cite{hedberg}) and conclude that $\sqrt{\vf(|u|)}\, u\in \Hspt^{1}(\Om)$.
The proof is complete  for $r\geq 0$.

If $-1<r<0$   we write   the function $\sqrt{\vf(|u|)}\, u$ as in \eqref{eq:phipsi}. What we have already proved
for $r\geq 0$ shows that
$\sqrt{\psi(|w|)}\, w \in H^{1}(\Om)$ ($\Hspt^{1}(\Om)$). 
\end{proof}

% The function $\vf(|u|)\, u$ belongs to $\Cspt(\Om)$ and
% $$
% \nabla(\vf(|u|)\, u)= \vf'(|u|)\Re(|u|^{{-1}}\overline{u}\, \nabla 
% u)\, u + \vf(|u|)\, \nabla u
% $$
% on $\{x\in\Om\ |\ u(x)\neq 0\}$. The continuity of $\nabla(\vf(|u|)\, 
% u)$ follows from the inequality
% $$
% |\nabla(\vf(|u|)\, u)| \leq (|\vf'(|u|)|\, |u|+\vf(|u|))\, |\nabla u|
% \lesssim |u|^{{r}}|\nabla u|
% $$
% (see \eqref{eq:condsfi}). If $r=0$ the condition \ref{item4} 
% required for $\vf$ shows that 
% $$
% \vf'(|u|)\Re(|u|^{{-1}}\overline{u}\, \nabla 
% u)\, u  + \vf(|u|)\, \nabla u \to \vf_{+}(0)\, \nabla u
% $$ 
% as $|u|\to 0$ and then
% $\vf(|u|)\, u\in \Cspt^{1}(\Om)$.

\section{A lemma on a bilinear form}\label{sec:lemma}

This Section is devoted to prove a Lemma which we shall use in the main Theorem
and concerns the positivity of the  bilinear form $ \la  \BB \nabla v, \nabla v \ra$
in $\Cspt^{\infty}(\Om)\times \Cspt^{\infty}(\Om)$,
where $\BB =\{b_{hj}\}$ is a real matrix whose elements depend on $x$ and $| v |$. We note that  if 
$b_{hj}$ do not depend on $|v  |$,  the result is well known
and can be obtained by standard arguments (see, e.g., \cite[pp.107--108]{fichera}).

\begin{lemma}\label{lemma:BB}
    Let $\BB =\{b_{hj}\}$ a real matrix whose elements belong to 
$L^{\infty}(\Om\times\R^{+})$ and $b_{hj}(x,t)$ are continuous with respect to $t\in \R^{+}$.
 If
 \begin{equation}\label{1}
\int_{\Om}\la  \BB(x,|v |) \nabla v, \nabla v\ra dx \geq 0
\end{equation}
for any real valued scalar function $v\in\Cspt^{\infty}(\Om)$, then
\begin{equation}\label{2}
\la  \BB(x,t)\, \xi, \xi\ra \geq 0
\end{equation}
for almost every $x \in \Om$ and for any $t>0$, $\xi\in \R^{N}$.
\end{lemma}

Let us assume that the matrix $\BB$ does not depend on $x$
and suppose that \eqref{2} is false.
This means that there exists $t_{0}>0$ and $\om_{0}$, $|\om_{0}|=1$ such 
that
$
\la \BB(t_{0})\, \om_{0},\om_{0}\ra <0 .
$

We can find two positive constants $M,\de$ such that 
\begin{equation}\label{ipotass}
\la \BB(t)\, \om,\om \ra \leq -M, \quad \forall \ t, \om: 
|t-t_{0}|<\de, |\om - \om_{0}|<\de, 
|\om|=1.
\end{equation}

 Without loss of generality we assume that the origin belongs to $\Om$.
Let now $v(x)=\be(\ro)\, \ga(\om)$ (where, as usual, $x=\ro\, \om$ ($\ro>0,\, |\om|=1$)), where $\be\in\Cspt^{\infty}(\R^{+})$ and 
$\ga\in C^{\infty}(\Si)$, $\Si$ being the unit sphere 
in $\R^{N}$. Note that
$$
\nabla v(x) = \dot{\be}(\ro)\, \ga(\om)\, \om + \be(\ro)\, \ro^{-1}
\nabla_{\om}\ga(\om)
$$
where the dot and $\nabla_{\om}$ denote the derivative with respect  
to $\ro$ and the tangential gradient 
\footnote{
The tangential gradient $\nabla_{\om}$ of a function $u$ can be defined as
$$
\nabla_{\om}  u = \ro \left( \nabla u  - \frac{\partial u}{\partial \ro}\, \om\right).
$$

By introducing local coordinates on the sphere of radius $\ro$, one can verify that
$\nabla_{\om}$ is a tangential operator acting on $\om$ and that it does not depend on $\ro$.}
on the unit 
sphere respectively.

Assuming that the support of $\be$ is so small that $\spt v \subset\Om$, we have
\begin{gather*}
	0 \leq \int_{\Om}\la  \BB(|v|) \nabla v, \nabla v \ra\, dx = 
	\int_{\Om}\la  \BB(|\be(\ro)\, \ga(\om)|)\, \om, \om \ra\, 
	\dot{\be}^{2}(\ro)\, \ga^{2}(\om)\,dx	+\\
	\int_{\Om}\la  [\BB(|\be(\ro)\, \ga(\om)|) + 
	\BB^{*}(|\be(\ro)\, \ga(\om)|)]\, \om, \nabla_{\om}\ga(\om) \ra\, 
	\dot{\be}(\ro)\, \be(\ro)\, \ro^{-1} \ga(\om)\,dx + \\
	\int_{\Om}\la  \BB(|\be(\ro)\, \ga(\om)|)\, 
	\nabla_{\om}\ga(\om), \nabla_{\om}\ga(\om) \ra\, 
	 \be^{2}(\ro)\,\ro^{-2} dx\, ,
\end{gather*}
i.e.
\begin{equation}
\begin{gathered}
0 \leq \int_{0}^{+\infty}
\dot{\be}^{2}(\ro)\,\ro^{N-1}d\ro\int_{|\om|=1} \la \BB(|\be(\ro)\, \ga(\om)|)\, \om, \om \ra\, 
\ga^{2}(\om)\, d\si_{\om} + \\
\int_{0}^{+\infty} 
\dot{\be}(\ro)\, \be(\ro)\, \ro^{N-2}d\ro\int_{|\om|=1}
\la  [\BB(|\be(\ro)\, \ga(\om)|) + \\
	\BB^{*}(|\be(\ro)\,\ga(\om)|)]\, \om, \nabla_{\om}\ga(\om) \ra
\ga(\om)\, d\si_{\om} + \\
\int_{0}^{+\infty}
\be^{2}(\ro)\,\ro^{N-3}d\ro\int_{|\om|=1} \la \BB(|\be(\ro)\, \ga(\om)|)\, \nabla_{\om}\ga(\om), 
\nabla_{\om}\ga(\om) \ra\, d\si_{\om}\, .
\label{coordpol}
\end{gathered}
\end{equation}

We choose now a particular  sequence of test functions. Fix $0<\ro_1 < \ro_2 < \ro_3 < \ro_4
< \text{dist}\,(0,\partial\Om)$.
The sequences $\{\be_{m}(\ro)\}$ and  $\{\ga_{m}(\om)\}$ are  required to satisfy the following conditions:
\begin{equation}
\begin{cases}
\be_{m}\in\Cspt^{\infty}(\R^{+}),\quad \text{spt}\, \be_{m}\subset 
(\ro_1,\ro_4)\, ;\\
\be_{m}(\ro)=\be_{1}(\ro), \quad \ro\in (\ro_1,\ro_4)\setminus (\ro_2,\ro_3),\ 
m=1,2,\ldots\, ;\\
t_{0}-\de < \be_{m}(\ro) < t_{0}+ \de, \quad \ro\in (\ro_2,\ro_3),\  m=1,2,\ldots\, ;\\
\displaystyle\lim_{n\to\infty}\int_{\ro_2}^{\ro_3}
\dot{\be}_{m}^{2}(\ro)\, \ro^{N-1}d\ro = +\infty\, ;
\label{condpsi}
\end{cases}
\end{equation}

\begin{equation}
\begin{cases}
	\ga_{m}\in C^{\infty}(\Si),\, 
	\text{spt}\,\ga_{m} \subset \Si_{2\de}\, ;\\
\quad 0\leq \ga_{m}\leq 1;\quad \ga_{m}(\om)=1, 
\forall\,\om\in \Si_{\de}\, ;\\
\displaystyle 
\int_{\Si_{2\de}\setminus\Si_{\de}}\ga_{m}^{2}(\om)\, d\si_{\om} =
{\cal O}\left(1/\sqrt{\lambda_{m}}\right)
\, ;\\
\displaystyle
\int_{\Si_{2\de}\setminus\Si_{\de}}|\nabla_{\om}\ga_{m}(\om)|^{2}\, d\si_{\om} =
{\cal O}\left(\sqrt{\lambda_{m}}\right)
\, .
\label{condga}
\end{cases}
\end{equation}

Here $\Si_{\de}$ is the set $\{\om\in \Si\ |\ |\om-\om_{0}| < \de\}$
and
$$
\lambda_{m}=\int_{\ro_2}^{\ro_3}
\dot{\be}_{m}^{2}(\ro)\, \ro^{N-1}d\ro\, .
$$

As a consequence we have also
\begin{equation}\label{ganabga}
\begin{gathered}
\int_{\Si_{2\de}\setminus\Si_{\de}}
|\ga_{m}(\om)|\, |\nabla_{\om}\ga_{m}(\om)|\, d\si_{\om} \leq \\
\left( \int_{\Si_{2\de}\setminus\Si_{\de}}  
\ga_{m}^2(\om)    d\si_{\om}
\right)^\frac{1}{2}
\left(\int_{\Si_{2\de}\setminus\Si_{\de}}   
 |\nabla_{\om}\ga_{m}(\om)|^2 d\si_{\om}
   \right)^\frac{1}{2} =
{\cal O}\left(1\right).
\end{gathered}
\end{equation}

Inequality \eqref{coordpol} and conditions \eqref{condpsi},
\eqref{condga} imply
\begin{gather*}
0 \leq\int_{\ro_2}^{\ro_3}
\dot{\be}_{m}^{2}(\ro)\, \ro^{N-1}d\ro \int_{\Si_{\de}}\la 
\BB(|\be_{m}(\ro)|)\, \om, \om \ra\, 
 d\si_{\om} + \\
\int_{\ro_2}^{\ro_3}
\dot{\be}_{m}^{2}(\ro)\, \ro^{N-1}d\ro 
\int_{\Si_{2\de}\setminus\Si_{\de}}\la 
\BB(|\be_{m}(\ro)| \ga_{m}(\om))\, \om, \om \ra\, \ga_{m}^{2}(\om)\,
 d\si_{\om} +
 \\
  \int_{(\ro_1,\ro_4)\setminus(\ro_2,\ro_3)}
\dot{\be}_{1}^{2}(\ro)\, \ro^{N-1}d\ro \int_{\Si_{\de}}\la 
\BB(|\be_{1}(\ro)|)\, \om, \om \ra\, 
 d\si_{\om} + \\
\int_{(\ro_1,\ro_4)\setminus(\ro_2,\ro_3)}
\dot{\be}_{1}^{2}(\ro)\, \ro^{N-1}d\ro \int_{\Si_{2\de}\setminus\Si_{\de}}\la 
\BB(|\be_{1}(\ro)| \ga_{m}(\om))\, \om, \om \ra\, \ga_{m}^{2}(\om)\,
 d\si_{\om}+ \\
 \int_{\ro_2}^{\ro_3}
\dot{\be}_{m}(\ro)\, \be_{m}(\ro)\, \ro^{N-2}d\ro 
\int_{\Si_{2\de}\setminus\Si_{\de}}\la 
[\BB(|\be_{m}(\ro)| \ga_{m}(\om) )+ \\
	\BB^{*}(|\be_{m}(\ro)\,\ga_{m}(\om)|)]\, \om, \nabla_{\om}\ga_{m}(\om) \ra\, \ga_{m}(\om)\,
 d\si_{\om} +
 \\
 \int_{(\ro_1,\ro_4)\setminus(\ro_2,\ro_3)}
\dot{\be}_{1}(\ro)\, \be_{1}(\ro)\, \ro^{N-2}d\ro
\int_{\Si_{2\de}\setminus\Si_{\de}}\la 
[\BB(|\be_{1}(\ro)| \ga_{m}(\om)+ \\
	\BB^{*}(|\be_1(\ro)\,\ga_{m}(\om)|)]\, \om, \nabla_{\om}\ga_{m}(\om) \ra\, \ga_{m}(\om)\,
 d\si_{\om} +\\
    \int_{\ro_1}^{\ro_4}
\be^{2}_{m}(\ro)\, \ro^{N-3}d\ro 
\int_{\Si_{2\de}\setminus\Si_{\de}}\la 
\BB(|\be_{m}(\ro)| \ga_{m}(\om))\, \nabla_{\om}\ga_{m}(\om), 
\nabla_{\om}\ga_{m}(\om) \ra \,
 d\si_{\om}\, .
 \end{gather*}

It is worth noting that 
\begin{gather*}
 \int_{\ro_2}^{\ro_3}
\dot{\be}_{m}(\ro)\, \be_{m}(\ro)\, \ro^{N-2}d\ro \leq  \\
\left( \int_{\ro_2}^{\ro_3} \dot{\be}_{m}^{2} (\ro)\, \ro^{N-2}d\ro \right)^\frac{1}{2}
\left(\int_{\ro_2}^{\ro_3}  \be_{m}^{2}(\ro) \, \ro^{N-2}d\ro
  \right)^\frac{1}{2}.
\end{gather*}

Since the sequence $\{\be_{m}\}$ is uniformly bounded,
the matrix $\BB$ is bounded and recalling condition \eqref{ipotass}, we deduce
that there exists a constant $K$ such that
	\begin{gather*}
0\leq -M\, |\Si_{\de}|\, \lambda_{m} +
K\bigg(\lam_{m}\int_{\Si_{2\de}\setminus\Si_{\de}} \ga_{m}^{2}(\om)\,
 d\si_{\om} 
 +1 + \int_{\Si_{2\de}\setminus\Si_{\de}} \ga_{m}^{2}(\om)\,
 d\si_{\om} 
+ \\
(\sqrt{\lam_{m}}+1)\int_{\Si_{2\de}\setminus\Si_{\de}} 
|\ga_{m}(\om)|\, |\nabla_{\om}\ga_{m}(\om)|\, d\si_{\om} +
\int_{\Si_{2\de}\setminus\Si_{\de}} 
|\nabla_{\om}\ga_{m}(\om)|^{2} d\si_{\om}\bigg).
\end{gather*}

Dividing by $\lam_{m}$ we get
	\begin{gather*}
0\leq -M \, |\Si_{\de}| +
K\bigg(\int_{\Si_{2\de}\setminus\Si_{\de}} \ga_{m}^{2}(\om)\,
 d\si_{\om} 
 +\lam_{m}^{-1}\left(1 + \int_{\Si_{2\de}\setminus\Si_{\de}} \ga_{m}^{2}(\om)\,
 d\si_{\om} \right)
+ \\
(\lam_{m}^{-1/2}+\lam_{m}^{-1})\int_{\Si_{2\de}\setminus\Si_{\de}} 
|\ga_{m}(\om)|\, |\nabla_{\om}\ga_{m}(\om)|\, d\si_{\om} +
\lam_{m}^{-1}\int_{\Si_{2\de}\setminus\Si_{\de}} 
|\nabla_{\om}\ga_{m}(\om)|^{2} d\si_{\om}\bigg).
\end{gather*}

Letting $m\to\infty$ and keeping in mind 
\eqref{condga} and \eqref{ganabga} we obtain
$$
0\leq -M\, |\Si_{\de}|
$$
and this is absurd. 
Inequality \eqref{2} is then proved when the matrix $\BB$ does not 
depend on $x$. 

In the general case, suppose \eqref{1} holds and take
$$
v(x) = w((x-x_{0})/\ep)
$$
where $x_{0}\in\Om$ is a fixed point, $w\in \Cspt^{\infty}(\Om)$ and $\ep>0$ is sufficiently small.
In this case \eqref{1} shows that
\begin{gather*}
	0\leq \frac{1}{\ep^{2}}\int_{\Om}\la  \BB(x,|w((x-x_{0})/\ep)|) \nabla w((x-x_{0})/\ep), 
	\nabla w((x-x_{0})/\ep)\ra\, dx = \\
	\ep^{N-2}\int_{\Om}\la  \BB(x_{0}+\ep y,|w(y)|) \nabla 
	w(y), 
	\nabla w(y)\ra\, dy\, .
\end{gather*}
 Therefore
 \begin{gather*}
\int_{\Om}\la  \BB(x_{0},|w(y)|) \nabla 
	w(y), 
	\nabla w(y)\ra\, dy = \\
	\lim_{\ep\to 0^{+}}\int_{\Om}\la  \BB(x_{0}+\ep y,|w(y)|) \nabla 
	w(y), 
	\nabla w(y)\ra\, dy \geq 0
\end{gather*}
for almost any $x_{0}\in\Om$. The arbitrariness of $w\in\Cspt^{\infty}(\Om)$ and what we have obtained for 
matrices not depending on $x$ give the result.

\section{The functional dissipativity}\label{sec:funcdiss}
 Let $\Om$ be a domain in $\R^N$, $\vf$ a function satisfying the conditions 1--4 in Section \ref{sec:phi} and
    $Eu=\nabla(\A \nabla u)$. 

The next Theorem provides a necessary and sufficient condition for the  functional dissipativity of the operator $E$ .

\begin{lemma}\label{lemma:fond}
   The operator $E$  is $L^{\Phi}$-dissipative if and only if
    \begin{equation}
\begin{gathered}\label{eq:condv}
\Re \int_{\Om}\Big[ \la \A\nabla v, \nabla v \ra +
\Lambda(|v|)\, \la (\A-\A^{*})\nabla |v|, |v|^{-1}\overline{v} \nabla v ) 
\ra +\\
-\Lambda^{2}(|v|)\, \la \A\nabla |v|, \nabla |v| \ra \Big] dx \geq 0,
\qquad \forall v \in \Hspt^{1}(\Om),
\end{gathered}
\end{equation}
where $\Lambda$ is given by  \eqref{eq:defHGA}. Here and in the sequel the integrand
is extended by zero on the set where $v$ vanishes.
\end{lemma}
\begin{proof}
	\textit{Sufficiency.}
	Suppose $r\geq 0$. Let $u\in \Hspt^{1}(\Om)$ such that $\vf(|u|)\, u\in \Hspt^{1}(\Om)$ and define $v=\sqrt{\vf(|u|)}\, u$. In view
	of   Lemma \ref{lemma:fipsi} we have that $v$ belongs to 
	$\Hspt^{1}(\Om)$.
	Moreover $u=|v|^{-1}\ze(|v|)\, v=\Theta(|v|)\, v$, $\vf(|u|)\, \overline{u} = |v|\, 
	[\ze(|v|)]^{-1}\overline{v}=[\Theta(|v|)]^{-1}v$ (see \eqref{eq:defHGA}). Therefore
	\begin{gather*}
\la \A \nabla u, \nabla(\vf(|u|)\, u)\ra =
%\la \A \nabla (|v|^{-1}\ze(|v|)\, v), \nabla(|v|\, 
%	[\ze(|v|)]^{-1}v)\ra =\\
	\la \A \nabla(\Theta(|v|)\, v), \nabla([\Theta(|v|)]^{-1}v)\ra =\\
	\la \A(\Theta'(|v|) v \nabla|v| + \Theta(|v|) \nabla v,
	- \Theta'(|v|) [\Theta(|v|)]^{-2} v\nabla|v|  + [\Theta(|v|)]^{-1} \nabla 
	v \ra  =\\
	 -(\Theta'(|v|)[\Theta(|v|)]^{-1}|v|)^{2}\la \A \nabla |v|, \nabla |v|\ra
	 +\\
	 \Theta'(|v|)[\Theta(|v|)]^{-1}(v \la \A \nabla |v|, \nabla v\ra - 
	 \overline{v}\la \A \nabla v, \nabla |v|\ra) + 
	 \la \A \nabla v, \nabla v\ra =\\
	 -\Lambda^{2}(|v|)\la \A \nabla |v|, \nabla |v|\ra + \\
	 \Lambda(|v|)(\la \A \nabla |v|, |v|^{-1}\overline{v}\,\nabla v\ra -
	 \la |v|^{-1}\overline{v}\, \nabla v, \A^{*}\nabla |v|\ra)
	 + \la \A \nabla v, \nabla v\ra .
\end{gather*}
on the set $\{x\in \Om\ |\ u(x)\neq 0\}=\{x\in \Om\ |\ v(x)\neq 0\}$.
	
Therefore
\begin{gather*}
\Re \int_{\Om}\la \A \nabla u, \nabla(\vf(|u|)\, u)\ra\, dx =\\
\Re \int_{\Om}\Big[ \la \A\nabla v, \nabla v \ra +
\Lambda(|v|)\, \la (\A-\A^{*})\nabla |v|, |v|^{-1}\overline{v} \nabla v ) 
\ra +\\
-\Lambda^{2}(|v|)\, \la \A\nabla |v|, \nabla |v| \ra \Big] dx \geq 0
\end{gather*}
because of \eqref{eq:condv}, and \eqref{eq:defdiss0} is proved.

If $-1<r<0$, setting $w=\vf(|u|)\, u$, i.e. $u=\psi(|w|)\, w$, we  can write  condition \eqref{eq:defdiss0} as
$$
\Re \int_{\Om}\la \A^{*}\nabla w, \nabla(\psi(|w|)\, w)\ra dx\geq 0
$$
for any $w\in \Hspt^{1}(\Om)$ such that $\psi(|w|)\, w\in \Hspt^{1}(\Om)$.

Recalling Lemma \ref{lemma:new}, what we have already proved for $r\geq 0$ shows that this inequality holds 
 if
\begin{equation}
\begin{gathered}\label{eq:condvstar}
\Re \int_{\Om}\Big[ \la \A^{*}\nabla v, \nabla v \ra +
\widetilde{\Lambda}(|v|)\, \la (\A^{*}-\A)\nabla |v|, |v|^{-1}\overline{v} \nabla v ) 
\ra +\\
-\widetilde{\Lambda}^{2}(|v|)\, \la \A^{*}\nabla |v|, \nabla |v| \ra \Big] dx \geq 0,
\qquad \forall v \in \Hspt^{1}(\Om).
\end{gathered}
\end{equation}

Since $\widetilde{\Lambda}(|v|)=-\Lambda(|v|)$ (see \eqref{eq:GAtilde}), conditions \eqref{eq:condvstar}
coincides with \eqref{eq:condv} and the sufficiency is proved also for $-1<r<0$.

\textit{Necessity.}
Let $v\in \Cspt^{1}(\Om)$ and define $u_{\ep}=\Theta(g_{\ep})\, v$, where
$g_{\ep}=\sqrt{|v|^{2}+\ep^{2}}$.

The function $u_{\ep}$ and  $\vf(|u_{\ep}|)\, u_{\ep}$ belong to $\Cspt^{1}(\Om)$ and we have
\begin{equation}
\begin{gathered}\label{eq:formulona}
\la \A \nabla u_{\ep}, \nabla(\vf(|u_{\ep}|)\, u_{\ep}\ra =\\
\vf(|u_{\ep}|)\, \la \A \nabla u_{\ep}, \nabla u_{\ep}\ra +
\vf'(|u_{\ep}|)\la \A \nabla u_{\ep},  u_{\ep}\,\nabla(|u_{\ep}|) \ra
=\\
\vf[\Theta(g_{\ep})\, |v|]\,
\la \A (\Theta'(g_{\ep})\, v\, \nabla g_{\ep} + \Theta(g_{\ep})\nabla v) ,
\Theta'(g_{\ep})\, v\, \nabla  g_{\ep} + \Theta(g_{\ep})\nabla v \ra + \\
\vf'[\Theta(g_{\ep})\, |v|]\times \\
\la \A (\Theta'(g_{\ep})\, v\, \nabla g_{\ep} + \Theta(g_{\ep})\nabla v) ,
\Theta(g_{\ep})\, v\, 
(\Theta'(g_{\ep})\, |v|\, \nabla  g_{\ep} + \Theta(g_{\ep})\nabla |v|) \ra 
=\\
\vf[\Theta(g_{\ep})\, |v|]\Big\{ [\Theta'(g_{\ep})]^{2}|v|^{2}
\la \A \nabla g_{\ep},\nabla g_{\ep}\ra + \\
\Theta'(g_{\ep})\, \Theta(g_{\ep})\, [ \la \A \nabla g_{\ep}, \overline{v}\,  
\nabla v \ra +
\la \A (\overline{v}\,  
\nabla v), \nabla g_{\ep}  \ra] + \Theta^{2}(g_{\ep})\, \la \A \nabla v, 
\nabla v \ra \Big\} + \\
\vf'[\Theta(g_{\ep})\, |v|]\Big\{
\Theta(g_{\ep})[\Theta'(g_{\ep})]^{2}|v|^{3} \la \A \nabla g_{\ep},\nabla 
g_{\ep}\ra + \\
\Theta^{2}(g_{\ep})\, \Theta'(g_{\ep}) [ |v|^{2} \la \A \nabla g_{\ep},\nabla 
|v| \ra + |v| \la \A (\overline{v}\nabla v), \nabla g_{\ep} \ra ]+ \\
\Theta^{3}(g_{\ep}) \la \A (\overline{v}\nabla v), \nabla |v| \ra \Big\}
\, .
\end{gathered}
\end{equation}

Letting $\ep \to 0^{+}$ the right hand side tends to
\begin{equation}\label{eq:rhstt}
\begin{gathered}
\vf[\Theta(|v|)\, |v|]\, \Theta^{2}(|v|)\, \la \A \nabla v, \nabla v\ra
+ \\
\vf[\Theta(|v|)\, |v|]\, \Theta'(|v|)\, \Theta(|v|)\, \la \A \nabla |v|, \overline{v}\nabla v\ra
+ \\
\Theta(|v|)\big\{ \vf[\Theta(|v|)\, |v|]\,\Theta'(|v|) + \\
\vf'[\Theta(|v|)\, |v|]\, 
\Theta(|v|)\,  [\Theta'(|v|)\, |v| +
\Theta(|v|)]\big\} \la \A (\overline{v}\nabla v), \nabla |v| \ra + \\
\Theta'(|v|)\, |v|^{2}\big\{ \vf[\Theta(|v|)\, |v|]\,\Theta'(|v|) +\\
\vf'[\Theta(|v|)\, |v|]\, \Theta(|v|)\, [\Theta'(|v|)\, |v|+ \Theta(|v|)]\big\}
\la \A \nabla |v|, \nabla |v| \ra
\end{gathered}
\end{equation}
on the set $\Om_{0}=\{x\in \Om\ |\ v(x)\neq 0\}$.

In view of \eqref{eq:H2psi} and \eqref{eq:tH'} we have
\begin{gather*}
\vf[\Theta(|v|)\, |v|]\, \Theta^{2}(|v|) = 1, 
\ \vf[\Theta(|v|)\, |v|]\, \Theta'(|v|)\, \Theta(|v|) =  \Theta'(|v|)/\Theta(|v|), \\
\vf[\Theta(|v|)\, |v|]\,\Theta'(|v|) +
\vf'[\Theta(|v|)\, |v|]\, \Theta(|v|)\, [\Theta'(|v|)\, |v|+ \Theta(|v|)] =\\
-\Theta'(|v|)/\Theta^{2}(|v|).
\end{gather*}

Substituting these equalities in \eqref{eq:rhstt} and keeping in mind \eqref{eq:formulona}, we see that
\begin{gather*}
\lim_{\ep\to 0^{+}}\la \A \nabla u_{\ep}, \nabla(\vf(|u_{\ep}|)\, u_{\ep}\ra =\\
\la \A \nabla v, \nabla v\ra 
+ \Lambda(|v|)\, ( \la \A \nabla |v|, |v|^{-1} \overline{v}\nabla v\ra -
\la \A (|v|^{-1} \overline{v}\nabla v), \nabla |v| \ra) +\\
-\Lambda^{2}(|v|) \la \A \nabla |v|, \nabla |v| \ra
\end{gather*}
on $\Om_{0}$.

By using \eqref{eq:condfi},\eqref{eq:dispsi}  and  \eqref{eq:disH'} 
one can prove that each term in the last expression of \eqref{eq:formulona} can be
majorized by $L^{1}$ functions. Let us consider the first one:
$\vf[\Theta(g_{\ep})\, |v|] [\Theta'(g_{\ep})]^{2}|v|^{2}
\la \A \nabla g_{\ep},\nabla g_{\ep}\ra$.  Observing also that 
$|\nabla g_{\ep}|\leq |\nabla v|$, we get
\begin{equation}\label{eq:get}
	\begin{gathered}
|\vf[\Theta(g_{\ep})\, |v|] [\Theta'(g_{\ep})]^{2}|v|^{2}
\la \A \nabla g_{\ep},\nabla g_{\ep}\ra | \lesssim 
[\Theta(g_{\ep})\, |v|]^{r} \Theta^{2}(g_{\ep})g^{-2}_{\ep} |v|^{2} |\nabla 
v|^{2} \leq \\
\Theta^{2+r}(g_{\ep})\, |v|^{r} |\nabla 
v|^{2} \lesssim g_{\ep}^{-r}
\, |v|^{r} |\nabla 
v|^{2}.
\end{gathered}
\end{equation}

Moreover
$$
g_{\ep}^{-r}
\, |v|^{r} \leq
\begin{cases}
 C, & \text{if }  r>0 \\ 
 C\, |v|^{r}, & \text{if }  r\leq 0,
\end{cases}
$$
where the constant $C$ does not depend on $\ep$.
Since the function $|v|^{r}|\nabla v|^{2}\chi_{\Om_{0}}$
belong to $L^{1}(\Om)$ because $r>-1$ (see Langer
\cite[p.312]{langer}),
 we see that in any case the last term in \eqref{eq:get} can be
majorized by an $L^{1}$ function which does not depend on $\ep$. 
The other terms in \eqref{eq:formulona} can be estimated in 
a similar way.

By the Lebesgue dominated convergence theorem, we find
\begin{gather*}
\lim_{\ep\to 0^{+}}\int_{\Om}\la \A \nabla u_{\ep}, 
\nabla(\vf(|u_{\ep}|)\, u_{\ep}\ra dx =\\
\int_{\Om}(\la \A \nabla v, \nabla v\ra 
+ \Lambda(|v|)\, ( \la \A \nabla |v|, |v|^{-1} \overline{v}\nabla v\ra -
\la \A (|v|^{-1} \overline{v}\nabla v), \nabla |v| \ra) +\\
-\Lambda^{2}(|v|) \la \A \nabla |v|, \nabla |v| \ra)\, dx \, .
\end{gather*}

The left hand side 
being non negative (see \eqref{eq:defdiss0}), inequality
\eqref{eq:condv} holds for any $v\in \Cspt^{1}(\Om)$.

Let  now $v\in \Hspt^{1}(\Om)$ and 
$v_{n}\in \Cspt^{\infty}(\Om)$  such that
$v_{n}\to v$ in $H^{1}$ norm. 
Let us show that
\begin{equation}
\begin{gathered}\label{eq:limitn}
\lim_{n\to\infty}
\int_{\Om}(\la \A \nabla v_{n}, \nabla v_{n}\ra 
+ \Lambda(|v_{n}|)\, ( \la \A \nabla |v_{n}|, |v_{n}|^{-1} 
\overline{v_{n}}\nabla v_{n}\ra +\\
-
\la \A (|v_{n}|^{-1} \overline{v_{n}}\nabla v_{n}), \nabla |v_{n}| \ra) 
-\Lambda^{2}(|v_{n}|) \la \A \nabla |v_{n}|, \nabla |v_{n}| \ra)\, dx =\\
\int_{\Om}(\la \A \nabla v, \nabla v\ra 
+ \Lambda(|v|)\, ( \la \A \nabla |v|, |v|^{-1} \overline{v}\nabla v\ra -
\la \A (|v|^{-1} \overline{v}\nabla v), \nabla |v| \ra) +\\
-\Lambda^{2}(|v|) \la \A \nabla |v|, \nabla |v| \ra)\, dx \, .
\end{gathered}
\end{equation}

We may assume that
$v_{n} \to v$, $\nabla v_{n} \to \nabla v$ almost everywhere
in $\Om$. Denote by $\Om_{0n}$ and $\Om_{0}$ the
sets 
$\{x\in\Om\ | \ v_{n}(x)\neq 0\}$ and 
$\{x\in\Om\ |\  v(x)\neq 0\}$, respectively. 
As proved in \cite[p.1087–-1088]{cialmaz}, 
$$
\chi_{\Om_{0n}}|v_{n}|^{-1} \overline{v_{n}}\nabla v_{n}
\to \chi_{\Om_{0}} |v|^{-1} \overline{v}\nabla v
\quad \text{a.e. in } \Om\, .
$$
Because of the continuity of $\Lambda$ on $(0,\infty)$ 
and its boundedness (see 
\eqref{eq:G<1}), we deduce
$$
\chi_{\Om_{0n}}\Lambda(|v_{n}|)\, |v_{n}|^{-1} \overline{v_{n}}\nabla v_{n}
\to \chi_{\Om_{0}} \Lambda(|v|)\, |v|^{-1} \overline{v}\nabla v
\quad \text{a.e. in } \Om\, .
$$

The boundedness of $\Lambda$  also leads to
\begin{gather*}
\int_{G}|\Lambda(|v_{n}|)\, ( \la \A \nabla |v_{n}|, |v_{n}|^{-1} \overline{v_{n}}\nabla v_{n}\ra -
\la \A (|v_{n}|^{-1} \overline{v_{n}}\nabla v_{n}), \nabla |v_{n}| \ra) +\\
-\Lambda^{2}(|v_{n}|) \la \A \nabla |v_{n}|, \nabla |v_{n}| \ra
|\, dx 
\lesssim 
\int_{G}|\nabla v_{n}|^{2}dx
\end{gather*}
for any measurable set $G\subset\Om$. This inequality
easily implies
that the sequence
of functions 
\begin{gather*}
\la \A \nabla v_{n}, \nabla v_{n}\ra 
+ \Lambda(|v_{n}|)\, ( \la \A \nabla |v_{n}|, |v_{n}|^{-1} 
\overline{v_{n}}\nabla v_{n}\ra +\\
-
\la \A (|v_{n}|^{-1} \overline{v_{n}}\nabla v_{n}), \nabla |v_{n}| \ra) 
-\Lambda^{2}(|v_{n}|) \la \A \nabla |v_{n}|, \nabla |v_{n}| \ra
\end{gather*}
satisfies the conditions of the Vitali convergence Theorem 
(see, e.g., \cite[p.71]{cannarsa}).

This establishes \eqref{eq:limitn} and the result follows from \eqref{eq:condv}.
\end{proof}

The next Corollaries provide necessary and, separately,  sufficient
conditions for the functional dissipativity of the operator $E$.

\begin{corollary}\label{co:1}
    If the operator $E$  is $L^{\Phi}$-dissipative, we have
	\begin{equation}\label{eq:E>=0}
\la \Re \A (x) \xi,\xi\ra \geq 0
\end{equation}
for almost every $x\in\Om$ and for any  $\xi\in\R^{N}$.
\end{corollary}
\begin{proof}

	Given a function $v\in\Cspt^{1}(\Om)$, define
\begin{equation}\label{eq:defXY}
X= \Re(|v|^{-1}\overline{v} \nabla v), \quad
	Y = \Im(|v|^{-1}\overline{v} \nabla v)
\end{equation}
on the set $\{x\in\Om\ |\ v(x)\neq 0\}$. As 
in \cite[p.1074]{cialmaz}, we have
\begin{gather*}
 \Re \la \A\nabla v,\nabla v\ra=
\la \Re \A X,X\ra +\la \Re \A Y,Y\ra
+\la \Im(\A-\A^{t})X,Y\ra ,\\
\Re\la (\A-\A^{*})\nabla |v|,|v|^{-1}\overline{v}
\nabla v\ra =
\la \Im(\A-\A^{*})X,Y\ra ,
\\
\Re \la \A\nabla|v|,\nabla|v|\ra = \la \Re \A X,X\ra.
\end{gather*}

The operator being $L^{\Phi}$-dissipative, \eqref{eq:condv} 
holds and
we can write
\begin{equation}
\begin{gathered}\label{eq:inXY}
\int_{\Om}\{[1-\Lambda^{2}(|v|)] \la\Re \A X,X\ra 
+ \la\Re \A Y,Y\ra +\\
[1+\Lambda(|v|)] \la \Im \A X,Y\ra 
+
[1-\Lambda(|v|)] \la \Im \A^{*} X,Y\ra \} dx \geq 0\, .
\end{gathered}
\end{equation}

Set $v(x)=\ro(x)\, e^{i\lam\xi\cdot x}$ where 
$\ro\in \Cspt^{\infty}(\Om)$ is a
 real valued function, $\xi\in\R^{N}$ and 
 $\lambda\in\R$. Putting $v$ in \eqref{eq:inXY}
 we get
 \begin{gather*}
\int_{\Om}[1-\Lambda^{2}(|v|)] \la\Re \A \nabla\ro,\nabla\ro\ra dx
+ \lam^{2}\int_{\Om}\ro^{2}\la\Re \A \xi,\xi\ra dx+\\
\lam \int_{\Om}\{[1+\Lambda(|v|)] \la \Im \A \nabla\ro,\xi\ra 
+
[1-\Lambda(|v|)] \la \Im \A^{*} \nabla\ro,\xi\ra \}\ro\,  
dx \geq 0\, .
\end{gather*}

For the arbitrariness of $\lam$ we find
$$
\int_{\Om}\ro^{2}\la\Re \A \xi,\xi\ra dx \geq 0.
$$
 
This inequality holding for any 
real valued $\ro\in\Cspt^{\infty}(\Om)$,
we obtain the result.
 \end{proof}
 
\begin{corollary}\label{co:2}
If
\begin{equation}\label{eq:polsuf}
\begin{gathered}
{}[1-\Lambda^{2}(t)] \la\Re \A(x)\, \xi,\xi\ra 
+ \la\Re \A(x)\, \eta,\eta\ra +\\
[1+\Lambda(t)] \la \Im \A(x)\, \xi,\eta\ra 
+
[1-\Lambda(t)] \la \Im \A^{*}(x)\, \xi,\eta\ra
\geq 0
\end{gathered}
\end{equation}
for almost every $x\in\Om$ and for any $t>0, \xi, \eta \in \R^{N}$, 
the operator $E$  is $L^{\Phi}$-dissipative.
\end{corollary}
\begin{proof}
Let $v\in \Hspt^{1}(\Om)$ and define 
$X$ and $Y$ as in the proof of Corollary
\ref{co:1}.
Inequality \eqref{eq:polsuf} implies that
\eqref{eq:inXY} holds. As we know, this
means that \eqref{eq:condv}
is satisfied and the result follows from
Lemma \ref{lemma:fond}.
\end{proof}

\begin{corollary}\label{co:3}
If the operator $E$ has real coefficients
and satisfies condition \eqref{eq:E>=0} for almost every $x\in\Om$ and for any  $\xi\in\R^{N}$,
than it is $L^{\Phi}$-dissipative with respect to any $\Phi$.
\end{corollary}
\begin{proof}
It follows immediately from Corollary \ref{co:2}
and \eqref{eq:G<1}.
\end{proof}

\begin{remark}
We shall see later a class of operators for which the positiveness of polynomials \eqref{eq:polsuf} 
is also necessary for the $L^{\Phi}$-dissipativity. But there are no functions $\vf$ for which
the condition  \eqref{eq:polsuf}  is necessary. This is shown by the next example.
\end{remark}

\begin{example}
 Consider the operator $E$ in two independent
variables where the matrix of the coefficients is
	$$
	\A=\left(\begin{array}{cc}
	1 & i\g \\ -i\g & 1
	\end{array}\right)
	$$
	$\g$ being a real constant.
	The polynomial in $\xi,\eta$ in condition  \eqref{eq:polsuf} is given by
	$$
	{}[1-\Lambda^{2}(t)]|\xi|^2 + |\eta|^2 + 2\, \ga\, (\xi_2	\eta_1 - \xi_1 \eta_2).
	$$
	Writing this polynomial in the form
	$$
	(\eta_1+\ga\xi_2)^2 +(\eta_2 - \ga\xi_1)^2 +[1-\Lambda^{2}(t)-\ga^2]\, |\xi|^2
	$$
	it is clear that, if $|\ga|>1$, condition \eqref{eq:polsuf} cannot be satisfied for any $\xi,\eta\in\R^N$. However, the corresponding
	operator is the Laplacean, which is $L^{\Phi}$-dissipative for any $\vf$ (see Corollary \ref{co:3}).
\end{example}

The next results concerns  $\Phi$-strongly elliptic operators

\begin{lemma}
    Let $E$ be a\  $\Phi$-strongly elliptic operator. There exists a constant $\kappa$ such that
   for any complex valued $u\in H^1(\Om)$ such that
   $\vf(|u|)\, u \in H^1(\Om)$ we have
    \begin{equation}\label{eq:dinineq}
\Re  \la \A \nabla u, \nabla(\vf(|u|)\, u)\ra   \geq \kappa 
| \nabla(\sqrt{\vf(|u|)}\, u)|^{2}
\end{equation}
almost everywhere on the set $\{x\in \Om\ |\ u(x)\neq 0\}$.
\end{lemma}
\begin{proof}
Let us  define $v=\sqrt{\vf(|u|)}\, u$. 
By Lemma \ref{lemma:fipsi}, the function $v$ belongs to $H^1(\Om)$.
As in the proof of Lemma \ref{lemma:fond}, we find
\begin{gather*}
\Re \la \A \nabla u, \nabla(\vf(|u|)\, u)\ra\  =\\
\Re \Big[ \la \A\nabla v, \nabla v \ra +
\Lambda(|v|)\, \la (\A-\A^{*})\nabla |v|, |v|^{-1}\overline{v} \nabla v ) 
\ra  +\\
-\Lambda^{2}(|v|)\, \la \A\nabla |v|, \nabla |v| \ra \Big] 
\end{gather*}
on the set $\{x\in \Om\ |\ u(x)\neq 0\}=\{x\in \Om\ |\ v(x)\neq 0\}$.

This can be written as
\begin{gather*}
\Re \la \A \nabla u, \nabla(\vf(|u|)\, u)\ra  =\\
[1-\Lambda^{2}(|v|)] \la\Re \A X,X\ra 
+ \la\Re \A Y,Y\ra +\\
[1+\Lambda(|v|)] \la \Im \A X,Y\ra 
+
[1-\Lambda(|v|)] \la \Im \A^{*} X,Y\ra 
\end{gather*}
where $X$ and $Y$ are given by \eqref{eq:defXY}. Thanks to
the $\Phi$-strong ellipticity (see \eqref{eq:Phise}) we get
$$
\Re \la \A \nabla u, \nabla(\vf(|u|)\, u)\ra  \geq \kappa(|X|^{2} + |Y|^{2})
=\kappa\, |\nabla v|^{2}
$$
and \eqref{eq:dinineq} is proved.
\end{proof}

This lemma implies  the next Corollary 
(see Dindo\v{s} and Pipher \cite[Th. 2.4, pp.263--265]{dindos1} for a similar result
in the $L^p$ case).

\begin{corollary}
      Let $E$ be a\  $\Phi$-strongly elliptic operator. There exists a constant $\kappa$ such that
    for any nonnegative $\chi\in L^{\infty}(\Om)$ and any complex valued $u\in H^1(\Om)$ such that
   $\vf(|u|)\, u \in H^1(\Om)$ we have
      $$
\Re \int_{\Om} \la \A \nabla u, \nabla(\vf(|u|)\, u)\ra\, \chi(x) dx \geq \kappa \int_{\Om}
| \nabla(\sqrt{\vf(|u|)}\, u)|^{2} \chi(x)\, dx\, .
$$
\end{corollary}
\begin{proof}
It follows immediately from inequality \eqref{eq:dinineq}.
\end{proof}

\section{A necessary and sufficient condition}\label{sec:cnes}

The aim of this section is to give a necessary and sufficient
condition for the $L^{\Phi}$-dissipativity of the operator $E$.

\begin{theorem}\label{th:main}
  Let the matrix $\Im \A$ be symmetric, i.e. 
  $\Im \A^{t}=\Im \A$. 
  Then the operator $E$  is 
  $L^{\Phi}$-dissipative if, and only if, 
  \begin{equation}\label{eq:cmcond}
|s\, \vf'(s)| \, | \la\Im \A (x)\, \xi,\xi\ra | 
\leq 2\, \sqrt{\vf(s)\, [s\, \vf(s)]'} 
\, \la \Re \A(x) \, \xi,\xi\ra 
\end{equation}
for almost every $x\in\Om$ and for any  $s>0, \xi\in\R^{N}$.
\end{theorem}
\begin{proof}
\textit{Sufficiency.} 
Let us prove that 
\eqref{eq:cmcond} implies inequality \eqref{eq:polsuf}, which, 
for the simmetricity of $\Im\A$, becomes
\begin{equation}\label{eq:polsuf2}
{}[1-\Lambda^{2}(t)] \la\Re \A(x)\, \xi,\xi\ra 
+ \la\Re \A(x)\, \eta,\eta\ra + 
2\, \Lambda(t) \la \Im \A(x)\, \xi,\eta \ra 
\geq 0
\end{equation}
for almost every $x\in\Om$ and for any   $t>0,  \xi,\eta \in \R^N$.

Fix $x\in\Om$ in such a way \eqref{eq:cmcond} holds, $t>0$ and define
$$
\Sm(\xi,\eta) = \la\Re \A(x)\, \xi,\xi\ra  +
\la\Re \A(x)\, \eta,\eta\ra + \ga\, \la\Im \A(x)\, \xi,\eta\ra
$$
where  
$$
\ga =\frac{s\, \vf'(s)}{\sqrt{\vf(s)\, [s\, \vf(s)]'}}\, ,
\qquad s=\ze(t)\, ,
$$
($\ze$ is the funtcion introduced in section \ref{subsec:aux}).

Let
$$
\lam = \min_{|\xi|^{2}+|\eta|^{2}=1}\Sm(\xi,\eta)\, .
$$

There exist $(\xi_{0},\eta_{0})$ such that
$|\xi_{0}|^{2}+|\eta_{0}|^{2}=1$ and
$\lam = \Sm(\xi_{0},\eta_{0})$. This vector satisfies
the algebraic system
$$
\begin{cases}
 \Re(\A+\A^{t})\,\xi_{0} + \ga \Im\A \eta_{0}= 2\, \lam\, \xi_{0}  \\ 
 \Re(\A+\A^{t})\,\eta_{0} + \ga \Im\A \xi_{0}= 2\, \lam\, \eta_{0}\, . 
\end{cases}
$$
This implies
$$
\Re(\A+\A^{t})\,(\xi_{0}-\eta_{0})
-\ga \Im\A\, (\xi_{0}-\eta_{0}) = 2\, \lam\, (\xi_{0}-\eta_{0})\, .
$$
and therefore
\begin{equation}\label{eq:condalg}
2\,\la\Re\A\,(\xi_{0}-\eta_{0}), \xi_{0}-\eta_{0}\ra 
-\ga \la \Im\A\, (\xi_{0}-\eta_{0}), \xi_{0}-\eta_{0}\ra = 
2\, \lam\, |\xi_{0}-\eta_{0}|^{2} .
\end{equation}

The left hand-side is nonnegative because of \eqref{eq:cmcond}.
If $\lam <0$, \eqref{eq:condalg} implies $\xi_{0}=\eta_{0}$.
In this case
$$
\lam = \Sm(\xi_{0},\xi_{0})=
2\, \la\Re \A(x)\, \xi_{0},\xi_{0}\ra + \ga\, \la\Im \A(x)\, 
\xi_{0},\xi_{0}\ra\, ;
$$
but this is nonnegative because of \eqref{eq:cmcond} and we get a 
contradiction. Therefore $\lam\geq 0$ and 
$\Sm(\xi,\eta)\geq 0$, for any $\xi,\eta\in\R^{N}$.

We have also $\Sm(-\sqrt{1-\Lambda^{2}(t)}\, 
\xi,\eta)\geq 0$, i.e.
\begin{equation}\label{eq:conxiG}
{}[1-\Lambda^{2}(t)] \la\Re \A(x)\, \xi,\xi\ra 
+ \la\Re \A(x)\, \eta,\eta\ra - 
\ga\, \sqrt{1-\Lambda^{2}(t)}\, \la \Im \A(x)\, \xi,\eta \ra 
\geq 0.
\end{equation}

On the other hand, \eqref{eq:G=} and \eqref{eq:1-G2} show that
$$
\ga\, \sqrt{1-\Lambda^{2}(t)} = \frac{2\,s\, \vf'(s)}{s\,\vf'(s)+2\, 
\vf(s)} = - 2\, \Lambda(t)
$$
and then \eqref{eq:conxiG} coincides with \eqref{eq:polsuf2}.

Corollary \ref{co:2} shows that the operator $E$ is $L^{\Phi}$-dissipative.

\textit{Necessity.} As in the proof of Corollary \ref{co:1},
the $L^{\Phi}$-dissipativity of $E$ implies 
$$
\int_{\Om}\{[1-\Lambda^{2}(|v|)] \la\Re \A X,X\ra 
+ \la\Re \A Y,Y\ra +
2\, \Lambda(|v|) \la \Im \A X,Y\ra 
 \} dx \geq 0\, ,
$$
for any $v\in\Cspt^{1}(\Om)$ (see \eqref{eq:inXY}).
Setting $v(x)=\ro(x)\,e^{i\, \si(x)}$, where $\ro\in\Cspt^{\infty}(\Om)$ 
and $\si\in C^{\infty}(\Om)$ are real valued, we get
$
|v|^{-1}\overline{v} \nabla v = |\ro|^{-1}\ro\, \nabla \ro + i\, 
|\ro|\, \nabla \si
$ on the set $\{x\in\Om\ |\ \ro(x)\neq 0\}$. 
It follows
\begin{equation}\label{eq:takes}
\begin{gathered}
\int_{\Om}\{[1-\Lambda^{2}(|\ro|)] \la\Re \A \nabla \ro, \nabla \ro\ra
+ \ro^{2}\la\Re \A \nabla \si, \nabla \si\ra
+ \\
2\, \Lambda(|\ro|) \, \ro\, \la \Im \A \nabla \ro, \nabla \si\ra 
 \} dx \geq 0
\end{gathered}
\end{equation}
for any real valued $\ro\in\Cspt^{\infty}(\Om), \si\in C^{\infty}(\Om)$.

We choose $\si$ by the equality
$$
\si(x)=\frac{\mu}{2}\log(\ro^{2}+\ep^{2})
$$
where $\mu\in\R$ and $\ep>0$. Inequality \eqref{eq:takes} takes the 
form
\begin{gather*}
\int_{\Om}\{[1-\Lambda^{2}(|\ro|)] \la\Re \A \nabla \ro, \nabla \ro\ra
+ \mu^{2}\frac{\ro^{4}}{(\ro^{2}+\ep)^{2}}\la\Re \A \nabla \ro, \nabla \ro\ra
+ \\
2\, \mu\, \frac{\ro^{2}}{\ro^{2}+\ep} \Lambda(|\ro|) \, \la \Im \A 
\nabla \ro, \nabla \ro\ra 
 \} dx \geq 0.
\end{gather*}

Letting $\ep\to 0^{+}$ we find
\begin{gather*}
\int_{\Om}\{[1-\Lambda^{2}(|\ro|)] \la\Re \A \nabla \ro, \nabla \ro\ra
+ \mu^{2} \la\Re \A \nabla \ro, \nabla \ro\ra
+ \\
2\, \mu\,  \Lambda(|\ro|) \, \la \Im \A 
\nabla \ro, \nabla \ro\ra 
 \} dx \geq 0.
\end{gather*}

This inequality holding for any real valued 
$\ro\in\Cspt^{\infty}(\Om)$, by Lemma \ref{lemma:BB}  we get
$$
[1-\Lambda^{2}(t)] \la\Re \A(x)\,  \xi,  \xi\ra
+ \mu^{2} \la\Re \A(x)\,  \xi,  \xi\ra
+ \\
2\, \mu\,  \Lambda(t) \, \la \Im \A(x)\, 
 \xi,  \xi\ra \geq 0
$$
for almost every $x\in\Om$ and for any  $t>0,  \xi\in\R^{n}$. The arbitrariness of $\mu\in\R$
leads to
$$
\Lambda^{2}(t) \la \Im \A(x)\, 
 \xi,  \xi\ra^{2} \leq  [1-\Lambda^{2}(t)]
 \, \la\Re \A(x)\,  \xi,  \xi\ra^{2}.
$$

Recalling \eqref{eq:1-G2} and Corollary \ref{co:1}, we can write
$$
|\Lambda(t)|\, | \la \Im \A(x)\, 
 \xi,  \xi \ra |  \leq  \sqrt{1-\Lambda^{2}(t)}\, 
 \, \la\Re \A(x)\,  \xi,  \xi\ra.
$$

Finally, setting $s=\ze(t)$  and keeping in mind the expressions \eqref{eq:G=} and \eqref{eq:1-G2}, 
the last inequality reads as
$$
 \frac{|s\, \vf'(s)|}{s\,\vf'(s)+2\, 
\vf(s)} | \la \, \Im \A(x)\, 
 \xi,  \xi\ra |  \leq 
 \frac{2\, 
\sqrt{\vf(s)\, (s\,\vf'(s)+\, \vf(s))}}{s\,\vf'(s)+2\, \vf(s)}
 \, \la\Re \A(x)\,  \xi,  \xi\ra ,
$$
 i.e. \eqref{eq:cmcond}.
\end{proof}

\begin{remark}
The proof of Theorem \ref{th:main}  shows that condition \eqref{eq:cmcond} holds if and only if
the inequality \eqref{eq:polsuf2} is satisfied for almost every $x\in\Om$ and for any $t>0,  \xi,\eta \in \R^N$. This
means that conditions \eqref{eq:polsuf} are necessary and sufficient for the $L^{\Phi}$-dissipativity
for the operators considered in Theorem \ref{th:main}.
\end{remark}

\begin{remark}
Suppose that the condition $\Im\A=\Im\A^t$ is not satisfied.
Arguing as in the proof of Theorem \ref{th:main}, one can prove that
condition \eqref{eq:cmcond} is still necessary for the $L^{\Phi}$-dissipativity of the operator $E$. 
However in general it is not sufficient, whatever the function $\vf$ may be.
This is shown by the next example. 
\end{remark}

\begin{example}
Let $n=2$ , $\Om$ a bounded domain and
$$
	\A=\left(\begin{array}{cc}
	1 & i\lam x_1 \\ -i\lam x_1 & 1
	\end{array}\right)
	$$

	Since $\la \Re\A \xi,\xi\ra = |\xi|^{2}$ and 
	$\la \Im\A \xi,\xi\ra = 0$ for any $\xi\in\R^{N}$, condition
	\eqref{eq:cmcond} is  satisfied.
	
If the corresponding operator $Eu=\Delta u + i\, \lam \, \partial_2 u$ 
is $L^{\Phi}$-dissipative,
then
\begin{equation}\label{eq:exfinale}
\Re \int_\Om \la \Delta u + i\, \lam \, \partial_2 u, u\ra\, \vf(|u|)\, 
dx \leq 0, \qquad \forall\ u\in \Cspt^\infty(\Om).
\end{equation}

Take $u(x)=\ro(x)\, e^{i\, t \, x_{2}}$, where $\ro\in 
\Cspt^\infty(\Om)$ is real valued and $t\in\R$.
Since $ \la Eu,u\ra = \ro[\Delta \ro + 2\, i\, t\, \partial_{2}\ro - 
t^{2}\ro
+ i\,\lam\,(\partial_{2}\ro + i t \ro)]$,
condition \eqref{eq:exfinale} implies
\begin{equation}\label{eq:imposs}
\int_{\Om}\ro\, \Delta \ro\, \vf(|\ro|)\, dx - \lam\, t \int_{\Om}\ro^{2} 
\vf(|\ro|)\, dx -t^{2}\int_{\Om}\ro^{2} 
\vf(|\ro|)\, dx \leq 0
\end{equation}
for any $t, \lam\in\R$. The function $\vf$ being positive, we can 
choose $\ro$ in such a way
$$
\int_{\Om}\ro^{2} 
\vf(|\ro|)\, dx >0.
$$

Taking  
$$
\lam^{2}>4 \int_{\Om}\ro\, \Delta \ro\, \vf(|\ro|)\, dx 
\left(\int_{\Om}\ro^{2} 
\vf(|\ro|)\, dx\right)^{-1},
$$
inequality \eqref{eq:imposs} is impossible for all $t\in\R$. Thus
$E$ is not $L^{\Phi}$-dissipative, although \eqref{eq:cmcond} is satisfied.

\end{example}

\begin{corollary}\label{co:4}
	Let the matrix $\Im \A$ be symmetric, i.e. 
  $\Im \A^{t}=\Im \A$. 
      If 
\begin{equation}\label{eq:lam0}
	  \lam_{0}=
	\sup_{s>0} 
	\frac{|s\, \vf'(s)|}
 	{2\, \sqrt{\vf(s)\, [s\, \vf(s)]'}} < +\infty, 
\end{equation}
then the operator $E$   is 
  $L^{\Phi}$-dissipative if, and only if, 
    \begin{equation}\label{eq:cmcond2}
\lam_{0}\,  | \la\Im \A (x)\, \xi,\xi\ra | 
\leq 
 \la \Re \A(x) \, \xi,\xi\ra 
\end{equation}
for almost every $x\in\Om$ and for any $ \xi\in\R^{N}$. If $\lam_{0}=+\infty$  the operator $E$  is $L^{\Phi}$-dissipative
 if and only if $\Im\A\equiv 0$ and condition \eqref{eq:E>=0}
is satisfied.
\end{corollary}
\begin{proof}
If $\lam_{0}<+\infty$, the result follows immediately from Theorem 
\ref{th:main}. 
If $\lam_{0}=+\infty$  and the operator $E$  is $L^{\Phi}$-dissipative, inequality \eqref{eq:cmcond} 
implies $\la\Im \A \xi,\xi\ra=0$, $\la \Re \A(x) \xi,\xi\ra\geq 0$ for almost every $x\in\Om$ and for any   $ \xi\in\R^{N}$.  
Therefore  $\Im \A\equiv 0$ and condition   \eqref{eq:E>=0}
is satisfied. The viceversa was proved in Corollary \ref{co:3}.
\end{proof}

\begin{remark}
If we use the function $\Phi$ (see \eqref{eq:Phiphi}), 
condition \eqref{eq:cmcond} can be written as
$$
|s\, \Phi''(s) - \Phi'(s)| \, | \la\Im \A (x)\, \xi,\xi\ra | 
\leq 2 \sqrt{s\, \Phi'(s)\, \Phi''(s)} 
\, \la \Re \A(x) \, \xi,\xi\ra 
$$
for almost every $x\in\Om$ and for any $s>0,  \xi\in\R^{N}$. In the same way, 
formula  \eqref{eq:lam0} becomes
$$	  \lam_{0}=
	\sup_{s>0} 
	\frac{|s\, \Phi''(s) - \Phi'(s)|}
 	{2 \sqrt{s\, \Phi'(s)\, \Phi''(s)}} < +\infty. $$
\end{remark}

We end this section by some examples in which we indicate both 
the functions $\Phi$ and $\vf$. It is easy to verify that in each example the
function $\vf$ satisfies conditions  \ref{item1}-\ref{item5} of section \ref{subsec:phi}.

\begin{example}
	If $\Phi(s)=s^p$, i.e. $\vf(s)=p\, s^{p-2}$, which corresponds to 
	$L^{p}$ norm, the function in \eqref{eq:lam0} is constant and
	$\lam_{0}=|p-2|/(2\sqrt{p-1})$. In this way we reobtain  Theorem 1 of
	\cite[p.1079]{cialmaz}.
\end{example}

\begin{example}
	Let us consider $\Phi(s)=s^{p}\log(s+e)$ ($p>1$), which is the
	Young function corresponding to the Zygmund space $L^{p}\, \log L$.
	This is equivalent to say $\vf(s)=p s^{p-2}\log(s+e) + 
	s^{p-1}(s+e)^{-1}$.
	By a direct computation we find
	\begin{equation}
\begin{gathered}
\frac{|s\, \Phi''(s) - \Phi'(s)|}
 	{2\sqrt{s\, \Phi'(s)\, \Phi''(s)}} =\\
\frac{\left| p(p-2)\log(s+e)+\frac{(2p-1)s}{s+e} - \frac{s^2}{(s+e)^{2}} \right| }
{2  \sqrt{\left(p\log(s+e) + \frac{s}{s+e}\right) \left(p(p-1)\log(s+e) + \frac{2ps}{s+e} - \frac{s^2}{(s+e)^{2}}\right) }} \, .
\label{eq:supPhi}
\end{gathered}
\end{equation}

Since
$$
\lim_{s\to 0^+}\frac{|s\, \Phi''(s) - \Phi'(s)|}
 	{2\sqrt{s\, \Phi'(s)\, \Phi''(s)}} = \lim_{s\to +\infty}\frac{|s\, \Phi''(s) - \Phi'(s)|}
 	{2\sqrt{s\, \Phi'(s)\, \Phi''(s)}} = \frac{|p-2|}{2\sqrt{p-1}}
$$
the function is bounded. Then we have  the 
$L^{\Phi}$-dissipativity of the operator $E$  
if, and only if, \eqref{eq:cmcond2} holds,
where  $\lam_{0}$ is the $\sup$ of the function  \eqref{eq:supPhi}
in $\R^{+}$.\end{example}

% \begin{example}
% 	Let us consider the Zygmund space  $L\, \log L$,
% 	which means taking $p=1$ in the previous example:. $\Phi(s)=s \log(s+e)$, $\vf(s)= s^{-1}\log(s+e) + 
% 	(s+e)^{-1}$.
% In this case
% 
% Since
% $$
% \lim_{s\to 0^+}\frac{|s\, \Phi''(s) - \Phi'(s)|}
%  	{2\sqrt{s\, \Phi'(s)\, \Phi''(s)}} = \lim_{s\to +\infty}\frac{|s\, \Phi''(s) - \Phi'(s)|}
%  	{2\sqrt{s\, \Phi'(s)\, \Phi''(s)}} = +\infty
% $$
% and then $\lam_{0}=+\infty$.  We have then the 
% $L^{\Phi}$-dissipativity of the operator $Eu=\nabla(\A\nabla u)$ 
% iif and only if the operator $E$ has real 
% coefficients and condition \eqref{eq:E>=0}
% is satisfied.
% \end{example}

\begin{example}
Let us consider the function $\Phi(s)=\exp(s^{p})-1$,
i.e. $\vf(s)=p\, s^{p-2} \exp(s^{p})$.
%, which is related
%to the Trudinger inequality
In this case
	$$\frac{|s\, \Phi''(s) - \Phi'(s)|}
 	{2\sqrt{s\, \Phi'(s)\, \Phi''(s)}} =
	\frac{|p\, s^p+ p -2|}{2 \sqrt{(p\, s^p +p -1)}}$$
and $\lam_{0}=+\infty$. In view of Corollary \eqref{co:4}, the operator $E$  is
$L^{\Phi}$-dissipative, i.e.
$$
	\Re \int_{\Om} \la \A\nabla u, \nabla[u\, |u|^{p-2} \exp(|u|^p) ]\ra dx \geq 0
	$$
 for any $u\in\Hspt^{1}(\Om)$ such that $|u|^{p-2} \exp(|u|^p)\, u\in\Hspt^{1}(\Om)$, if and only if the operator $E$ has real 
coefficients and condition \eqref{eq:E>=0}
is satisfied.
\end{example}

\begin{example}
	Let $\Phi(s)=s-\arctan s$, i.e. $\vf(s)=s/(s^{2}+1)$. In this case
	$$
	\frac{|s\, \Phi''(s) - \Phi'(s)|}
 	{2\sqrt{s\, \Phi'(s)\, \Phi''(s)}} =
	\frac{|s^{2}-1|}{2 \sqrt{2(s^{2}+1})}
	$$
	and $\lam_{0}=+\infty$. As in the previous example, we have that
	$$
	\Re \int_{\Om} \la\A\nabla u , \nabla\left(\frac{|u|\, u}{|u|^{2}+1}\right)\ra dx \geq 0
	$$
for any $u\in\Hspt^{1}(\Om)$ such that $|u|\, u /(|u|^{2}+1)\in\Hspt^{1}(\Om)$,  if and only if the operator $E$ has real 
coefficients and condition \eqref{eq:E>=0}
is satisfied.
\end{example}

\begin{example}
	Let $\Phi(s)=s^{4}/(s^{2}+1)$, i.e. $\vf(s)=2\,s^{2}(2+s^{2})/(s^{2}+1)^{2}$. In this case
	$$
	\frac{|s\, \Phi''(s) - \Phi'(s)|}
 	{2\sqrt{s\, \Phi'(s)\, \Phi''(s)}} =
	\frac{2}{\sqrt{(s^{2}+1)(s^{2}+2)(s^{4}+3s^{2}+6)}}\, .
	$$
	This function is decreasing and  $\lam_{0}$ is equal to its 
	value at $0$, i.e. $\lam_{0}=1/\sqrt{3}$. 
	The operator $E$  is
$L^{\Phi}$-dissipative, i.e.
	$$
	\Re \int_{\Om} \la \A\nabla u, \nabla\left(\, \frac{|u|^{2}(2+|u|^2)u}{(|u|^{2}+1)^{2}}\right)\ra dx \geq 0
	$$ 
	for any $u\in\Hspt^{1}(\Om)$ such that $|u|^{2}(2+|u|^2)u /(|u|^{2}+1)^{2} \in\Hspt^{1}(\Om)$, if and only 
	if
	$$
	| \la\Im \A (x)\, \xi,\xi\ra | 
\leq \sqrt{3} 
\, \la \Re \A(x) \, \xi,\xi\ra 
$$
for almost any $x\in\Om$ and for any $\xi\in\R^{N}$.
\end{example}

\begin{example}
	Let $\Phi(s)=s^{2}(s^{2}+2)/(s^{2}+1)-2 \log(s^{2}+1)$, i.e. 
	$\vf(s)=2\,s^{4}/(s^{2}+1)^{2}$. In this case
	$$
	\frac{|s\, \Phi''(s) - \Phi'(s)|}
 	{2\sqrt{s\, \Phi'(s)\, \Phi''(s)}} =
	\frac{2}{\sqrt{(s^{2}+1)(s^{2}+5)}}\, .
	$$
	This function is decreasing and  $\lam_{0}$ is equal to its 
	value at $0$, i.e. $\lam_{0}=2/\sqrt{5}$. 
	The operator $E$  is
$L^{\Phi}$-dissipative, i.e.
	$$
	\Re \int_{\Om} \la \A\nabla u, \nabla\left(\frac{|u|^{4}u}{(|u|^{2}+1)^{2}}\right)\ra dx \geq 0
	$$  
	for any $u\in\Hspt^{1}(\Om)$ such that $|u|^{4}u /(|u|^{2}+1)^{2} \in\Hspt^{1}(\Om)$,
	if and only 
	if
	$$
	2\, | \la\Im \A (x)\, \xi,\xi\ra | 
\leq \sqrt{5} 
\, \la \Re \A(x) \, \xi,\xi\ra 
$$
for almost any $x\in\Om$ and for any $\xi\in\R^{N}$.
\end{example}

{\bf Acknowledgements} The second author was supported by the RUDN University Program 5--100.

\end{document}